\documentclass[final]{siamltex}
\usepackage{amssymb}
\usepackage{amsmath}
\usepackage{booktabs}
\usepackage{multirow}
\usepackage{caption}
\usepackage{graphicx}
\usepackage{latexsym}
\usepackage{verbatim}
\usepackage{float}
\usepackage{subfig}
\usepackage{graphics}
\usepackage{epstopdf}
\usepackage{url}
\usepackage{subfig}
\usepackage{multirow}
\usepackage{color}
\usepackage{color}

\newtheorem{Lemma}[theorem]{Lemma}
\newtheorem{Corollary}[theorem]{Corollary}

\newcommand{\R}{\ensuremath{\mathbb{R}}}
 
\def\sinc{\mathrm{sinc}}

\title{A Fast Multipole Method based on Band-limited Approximations for Radial Basis Functions}  

\author{ Wei Zhao\thanks{Numerical Linear Algebra for Dynamical Systems,
Max Planck Institute for Dynamics of Complex Technical Systems, Sandtorstr. 1,
39106 Magdeburg, Germany ({\tt zhao@mpi-magdeburg.mpg.de})} 
\and Martin Stoll\thanks{Numerical Linear Algebra for Dynamical Systems,
Max Planck Institute for Dynamics of Complex Technical Systems, Sandtorstr. 1,
39106 Magdeburg, Germany ({\tt stollm@mpi-magdeburg.mpg.de})}}

\begin{document}
\maketitle

\begin{abstract}
The meshless/meshfree radial basis function (RBF) method is a powerful technique for interpolating scattered data. But,
solving large RBF interpolation problems without fast summation methods is computationally expensive. For RBF interpolation 
with $N$ points, using a direct method requires $\mathcal{O}(N^2)$ operations. As a fast summation method, the fast multipole
method (FMM) has been implemented in speeding up the matrix-vector multiply, which reduces the complexity from $\mathcal{O}(N^2)$ to $\mathcal{O}(N^{1.5})$ 
and even to $\mathcal{O}(NlogN)$ for the multilevel fast multipole method (MLFMM). In this paper, we present a novel kernel-independent fast multipole method for RBF interpolation, which is used in combination with the evaluation of point-to-point interactions by RBF and the fast matrix-vector multiplication. 
This approach is based on band-limited approximation and quadrature rules, which extends the range of applicability of FMM. 
\end{abstract}



\pagestyle{myheadings}
\thispagestyle{plain}
\markboth{}{}

\section{Introduction}\label{Intro}
The radial basis functions (RBF) method was successfully developed for scattered data approximation \cite{carr2001reconstruction,hardy1990theory,turk2002modelling} and has been applied in the numerical solution of partial differential equations (PDEs)\cite{hon2001unsymmetric,hon1999multiquadric,hon1998efficient,fedoseyev2002improved}. The main advantage of this method is that it is meshless/meshfree, i.e. no triangulation is needed. Other methods, e.g. finite element method (FEM), first generate a triangulation of the space, use functions on each component of the triangulation, and then patch them together obtaining a global function. The resulting function is not very smooth and the method suffers from the curse of dimensionality in higher space dimensions because generating the grid/mesh is time consuming \cite{griebel2000particle}. 
Let us explain the approximation with radial basis functions. Given a set of quasi-uniform~\cite{schaback1995error} centers $X=\{\mathbf{x}_1,\cdots,\mathbf{x}_N\}$ with the mesh norm $h$ and radial basis functions $\Phi$, the approximation has the form
\begin{equation}\label{eq:1}
 u(\mathbf{x}_i)=\sum_{j=1}^N\lambda_j\Phi(\mathbf{x}_i-\mathbf{x}_j),~~i=1,\cdots,N,
\end{equation}
where $\lambda_j$ are coefficients and $u$ is either the interpolation of a set of values or the numerical solution of a PDE. The corresponding theory has been studied in for example \cite{narcowich2002scattered,schaback1995error,wendland1998error,wendland2004scattered,buhmann2004radial}. This method requires $\mathcal{O}(N^2)$ complexity to evaluate the sums in \eqref{eq:1} using a direct summation method. When $N$ grows to be large, this approach will be prohibitively costly unless some fast summation methods can be considered.

There are three common fast summation algorithms including tree codes like Barnes-Hut~\cite{barnes1986hierarchical}, fast multipole method and fast convolution methods like FFT.  The fast multipole method (FMM) is a numerical algorithm introduced by Greengard and Rokhlin~\cite{greengard1987fast} for solving the potential field $u$ generated by a large number of unknown interactions. This method is based on the idea that one particle interacts with a group of other particles by approximating  their influence rather than interacting with each of them, when the group is the far-field (well-separated) of the particle. The description of the original FMM can be found~\cite{greengard1987fast,rokhlin1990rapid}. The key techniques of the FMM are expansions (multipole and local expansions) and translations (multipole-to-multipole, 
multipole-to-local and local-to-local translations). The FMM has been widely applied for many general kernels including the Laplace kernel $\frac{1}{\mathbf{r}}$~\cite{greengard1987fast,yarvin1999improved,hrycak1998improved}, the Helmholtz kernel $\frac{e^{i\mathbf{k}\mathbf{r}}}{\mathbf{r}}$ ~\cite{cecka2013fourier,engquist2010fast,darve2000fast,rokhlin1990rapid,rokhlin1993diagonal}, Stokes kernel~\cite{fu2000fast}, and Navier kernel~\cite{fu1998fast,yoshida2001application}. The application of the fast multipole method combined with some special radial basis functions has also been discussed in some papers. Beatson and Newsam~\cite{beatson1992fast} and Beatson and Greengard~\cite{beatson1997short} developed an FMM for the multiquadric(MQ) function using a far-field (multipole) Laurent series and a near-field (local) Taylor series, and Cherrie, Beatson, and Newsam~\cite{beatson1998fast} applied that approach for the generalized MQ function. In \cite{beatson1997fast} the authors expressed the MQ function as a 
Gaussian integral and applied quadrature rules and fast Gauss transform
(a special FMM)~\cite{greengard1991fast}. Beatson and Newsam~\cite{cherrie2002fast} as well as Livne and Wright~\cite{wright2006fast} proposed methods based on polynomial interpolation and multilevel summation. Gumerov and Duraiswami~\cite{gumerov2007fast} developed an FMM scheme for the 2D MQ function by relating it to the biharmonic kernel in 3D. In recent years, the range of applicability of the FMM has been extended by applying kernel-independent approaches. Ying~\cite{ying2004kernel} applied the kernel-independent FMM, which uses equivalent particles densities in place of analytic series expansions~\cite{ying2006kernel}. Another kernel-independent approach based on Cauchy's integral formula and the Laplace transform was proposed in ~\cite{letourneau2014cauchy}. For \eqref{eq:1}, the assumptions when applying the FMM are given by
\begin{itemize}
 \item The function $u(\cdot)$ occurs at {\it evaluation points} $\{\mathbf{x}_i\}$.
 \item Generally, the set of {\it source points} $\{\mathbf{x}_j\}$ and the set of evaluation points $\{\mathbf{x}_i\}$ contain about the same number of members.
 \item $\lambda_j$ are the {\it source~weights} and $\Phi$ is the {\it potential function}.
\end{itemize}
For a given precision, the FMM can accelerate the computation \eqref{eq:1} and reduce the complexity to $\mathcal{O}(N)$. All iterative Krylov methods for solving linear systems, such as CG method \cite{cg} or GMRES \cite{gmres} method, involve matrix-vector multiplications, therefore the FMM can speed up these iterative methods by replacing the matrix-vector products with applications of the FMM.
\par
The FMM consists of the following steps:
\begin{itemize}
 \item generation of  a hierarchical tree partitioning of the computational domain;
 \item evaluation of the multipole expansion for the far-field and aggregation of these contributions by a upward pass of the tree;
 \item translation of the multipole expansions to the local expansions;
 \item construction of local by downward pass of the tree;
 \item disaggregation of the contributions from far-field action on the particles by local expansions;
 \item evaluation of near-field interactions.
\end{itemize}
The same steps are also used in this paper. However before starting this algorithm, the potential function $u$ is replaced by a band-limited function $u_\sigma$ with the bandwidth $\sigma$.
According to the used quadrature rules, the band-limited function is expressed as a linear combination of exponential functions. Afterwards, all expansions and translations 
used in this paper rely on exponential functions such that they are related to a high frequency fast multipole method~\cite{burkholder1996high,cecka2013fourier}. 
\par
The organization of this paper is as follows. In Section 2, we provide the relevant mathematical background on radial basis functions and also briefly introduce the original fast multipole method based on~\cite{gumerov2006fast}. Section 3 introduces the band-limited approximations of the radial basis functions and the underlying theoretical analysis. Section 4 details the novel fast multipole method for radial basis functions based on band-limited approximation and also gives some numerical simulations.
\section{Mathematical Preliminaries}\label{subsection::Prel}
\subsection{Notation}
We start by introducing some notation.  For a bounded domain $\Omega\subseteq \R^d$ ($d$~is dimensions) and data point $X=\{\mathbf{x}_1,\cdots,\mathbf{x}_N\}\subseteq \Omega$, the mesh norm is defined as follows
\begin{equation}\label{eq:2}
 h=\sup_{\mathbf{x}\in\Omega}\min_{\mathbf{x}_j\in X}\|\mathbf{x}-\mathbf{x}_j\|_2.
\end{equation}
Moreover, for a non-negative integer $k$ and $1\leq p<\infty$ let $W_p^k(\Omega)$ denote the Sobolev space with differentiability order $k$ and integrability power $p$. Define for $u\in W_p^k(\Omega)$ and finite $p$ the Sobolev (semi-)norms
\begin{equation}\label{eq:3}
 |u|_{W_p^k(\Omega)}=\left(\sum_{|\alpha|=k}\|D^\alpha u\|_{L_p(\Omega)}^p\right)^{1/p}~~~and~~~\|u\|_{W_p^k(\Omega)}=\left(\sum_{|\alpha|\leq k}\|D^\alpha u\|_{L_p(\Omega)}^p\right)^{1/p}.
\end{equation}
In the case $p=2$, we have a Hilbert space and can introduce a norm via the Fourier transform, which has the advantage that it can be generalized to non-integer values $0<\tau<\infty.$ It then yields an equivalent norm to the one defined above if we choose $\tau$ to be an integer. We can describe the functions in the fractional Sobolev space $W_2^\tau(\R^d)$ as precisely square-integrable functions that are finite in the form
\begin{equation}\label{eq:4}
 \|u\|_{W_2^\tau(\R^d)}=\|(1+\|\mathbf{\omega}\|_2^2)^{\tau/2}\widehat{u}(\mathbf{\omega})\|_{L_2(\R^d)}
\end{equation}
Here, $\widehat{u}(\cdot)$ is the Fourier transform
\begin{equation}\label{eq:5}
 \widehat{u}(\mathbf{\omega})=\int_{\R^d}u(\mathbf{x})e^{-i\mathbf{\omega}\cdot\mathbf{x}}d\mathbf{x}.
\end{equation}
In this paper, we also use the inverse Fourier transform of the form 
\begin{equation}\label{eq:6}
 u(\mathbf{x})=(2\pi)^{-d}\int_{\R^d}\widehat{u}(\mathbf{\omega})e^{i\mathbf{\omega}\cdot\mathbf{x}}d\mathbf{x}.
\end{equation}
Let us now introduce the needed RBFs and their corresponding spaces.
\subsection{Radial Basis Functions and Native Space}
Let $\mathbf{r}=\|\cdot\|$ be the Euclidean norm on $\R^d$. A kernel function $\Phi(\mathbf{x},\mathbf{x}_j): \R^d \rightarrow \R$ with 
$\mathbf{x}_j=\{x_1,\cdots,x_d\}$ is called radial if
\begin{equation}\label{eq:7}
\Phi(\mathbf{x},\mathbf{x}_j)=\Phi(\mathbf{x}-\mathbf{x}_j)=\varphi(\|\mathbf{x}-\mathbf{x}_j\|)=\varphi(\mathbf{r}),~~x\in \R^d,
\end{equation}
For some univariate function $\varphi:[0,\infty)\rightarrow \R$. $\varphi(\mathbf{r})$ is used as a basis function in the RBF method and the univariate function $\varphi$ is independent from the number of dimensions $d$. Therefore, the RBF method can be easily adapted to solve higher dimensional problems. In recent applications, the RBFs most commonly used are given in Table \ref{tab::global} and Table \ref{tab::compact}.
\begin{table}[H]
\centering
 \begin{tabular}{ll} 
\toprule
  Gaussian (GA)                  & $e^{-c\mathbf{r}^2}$, $c>0$\\
  Multiquadric (MQ)              & $\sqrt{\mathbf{r}^2+c^2}$, $c>0$\\
  Inverse MQ                     & $1/\sqrt{\mathbf{r}^2+c^2}$, $c>0$\\
  Thin-plate spline (TPS)        & $(-1)^{1+\beta/2}\mathbf{r}^\beta\log \mathbf{r}$, $\beta\in 2N$\\
\bottomrule
\end{tabular}
\caption{Global functions}
\label{tab::global}
\end{table}

\begin{table}[H]
\centering
 \begin{tabular}{ll} 
\toprule
  $\Phi_{l,0}$                  & $(1-\mathbf{r})_+^l$\\
  $\Phi_{l,1}$                  & $(1-\mathbf{r})_+^{l+1}[(l+1)\mathbf{r}+1]$\\
  $\Phi_{l,2}$                  & $(1-\mathbf{r})_+^{l+2}[(l^2+4l+3)\mathbf{r}^2+(3l+6)\mathbf{r}+3]$\\
\bottomrule
\end{tabular}
\caption{ Compactly supported functions where $l=\lceil 2+k+1 \rceil$,$k=0,1,\ldots$.}
\label{tab::compact}
\end{table}

General convergence results for an RBF approximation on a domain $\Omega\in \R^d$ have been derived for functions on native spaces $\mathcal{N}_\Phi(\Omega)$ \cite{wendland2004scattered}. For strictly positive definite basis functions (SPD), such as Gaussian and IMQ, these spaces can be defined as the completion of the pre-Hilbert space
\begin{equation}\label{eq:8}
F_\Phi(\Omega):=span\{\Phi(\cdot,\mathbf{y}):\mathbf{y}\in\Omega\}
\end{equation}and we equip this space with the inner product
\begin{equation}\label{eq:9}
 (\sum_{i=1}^N\lambda_i\Phi(\cdot,\mathbf{x}_i),\sum_{j=1}^N\lambda_j\Phi(\cdot,\mathbf{x}_j))_{\Phi}:=\sum_{i,j=1}^N\lambda_i\lambda_j\Phi(\mathbf{x}_i-\mathbf{x}_j).
\end{equation}The native space for conditionally positive definite basis functions can be defined in a similar form~\cite{wendland2004scattered}. It is worth pointing out that the native space $\mathcal{N}_\Phi(\R^d)$ can be characterized using the Fourier transform,
\begin{equation}\label{eq:10}
\mathcal{N}_\Phi(\R^d):=\{f\in L_2(\R^d)\cap C(\R^d):\widehat{f}/\sqrt{\widehat{\Phi}}\in L_2(\R^d)\}.
\end{equation}We state the following result from \cite{wendland2004scattered}.
\begin{theorem}
 Suppose $\Phi\in L_1(\R^d)\cap C(\R^d)$ is radial, i.e. $\Phi(x)=\varphi(\|\mathbf{x}\|_2), x\in \R^d$. Then its Fourier transform
 $\widehat{\Phi}$ is also radial, i.e. $\widehat{\Phi}(\mathbf{\xi})=\mathcal{F}_d\varphi(\|\xi\|_2)$ with
 \begin{equation}\label{eq:11}
 \mathcal{F}_d\varphi(\mathbf{r})=(2\pi)^{d/2}\mathbf{r}^{-(d-2)/2}\int_0^\infty\varphi(t)t^{d/2}J_{(d-2)/2}(\mathbf{r}t)dt
 \end{equation}
\end{theorem}
From this the following useful result can be obtained (cf. \cite{wendland2004scattered}).
\begin{Corollary}
 Suppose that $\Phi$ satisfies
 \begin{equation}\label{eq:12}
 c_1(1+\|\mathbf{\xi}\|_2^2)^{-\tau}\leq\widehat{\Phi}(\mathbf{\xi})\leq c_2(1+\|\mathbf{\xi}\|_2^2)^{-\tau}, \mathbf{\xi}\in\R^d
 \end{equation}
 with $\tau>d/2$ and two positive constants $c_1\leq c_2$. Then the native space $\mathcal{N}_\Phi(\R^d)$ corresponding to
 $\Phi$ coincides with the Sobolev space $W_2^\tau(\R^d)$, and the native space norm and Sobolev norm are equivalent.
\end{Corollary}
The following interpolation error holds (see \cite{wendland1999meshless}).
\begin{Lemma}
 Let $\Omega\subseteq \R^d$ be an open and bounded domain, having a Lipschitz boundary and satisfying the interior cone condition.
 Denote by $u$ the interpolant on $X=\{x_1,x_2,\cdots,x_N\}\subset\Omega$ to a function $f\in W_2^\tau(\Omega), \tau>d/2$. Then there exists 
 a constant $h_0>0$ such that for all $X$ with $h<h_0$, where $h$ is the density of $X$, the estimate
 \begin{equation}\label{eq:error}
  \|f-u\|_{W_2^s (\Omega)}\leq Ch^{\tau-s}\|f\|_{W_2^\tau (\Omega)},~~~~~~~~~0\leq s\leq \tau.
 \end{equation}
\end{Lemma}
We now come to the introduction of our band-limited approximation.
\section{Approximation of Band-limited function}\label{subsection::BL}
According to a fundamental principle of the Fourier transform, smooth functions have Fourier transforms that decay rapidly to zero at infinity (see \cite{ForF05,wendland2004scattered}).
Radial basis functions give smooth approximations and their Fourier transforms fall into one of the following two cases
\begin{itemize}
\item If $\widehat{\Phi}(\xi)$ decays fast and tends to zero on a finite interval, then $\Phi(x)$ is a band-limited function. It can be fully
reconstructed from its samples and furthermore, the error decreases exponentially with bandwidth;
\item  If $\widehat{\Phi}(\xi)$ decays slowly as $|\xi|\rightarrow \infty$, a mollifier is introduced to accelerate the rate such that the 
mollification is a band-limited function. We  use this mollification as the approximation to replace $\Phi$.
\end{itemize}
We illustrate the connection between the RBFs and their corresponding Fourier transform in Table \ref{tab::fourier}.
\begin{table}[H]
\centering
\begin{tabular}{ll}
 \bottomrule
 RBFs $\Phi(r)$                & Fourier transform $\widehat{\Phi}(\xi)$\\
 \midrule
\multicolumn{2}{c}{Piecewise smooth}\\                
&\\
 $r^5$                                 &$\frac{-80\cdot2^{d/2}\Gamma(5/2)\Gamma((5+d)/2)}{\pi}\frac{1}{|\xi|^{5+d}}$\\
 $r^2\log r$                           &$2^{1+d/2}\Gamma(1+d/2)\frac{1}{|\xi|^{2+d}}$\\
\midrule                    
\multicolumn{2}{c}{Infinitely smooth}\\
&\\
 $\sqrt{1+r^2}$                        &$-\frac{\sqrt{2}}{\sqrt{\pi}}\frac{K_{(d+1)/2}(|\xi|)}{|\xi|^{(d+1)/2}}$\\
 $\dfrac{1}{\sqrt{1+r^2}}$                      &$\frac{\sqrt{2}}{\sqrt{\pi}}\frac{K_{(d-1)/2}(|\xi|)}{|\xi|^{(d-1)/2}}$\\
 $e^{-r^2}$                            &$\frac{e^{-\xi^2/4}}{(\sqrt{2})^d}$\\
 \bottomrule
\end{tabular}
\caption{\it (Generalized) Fourier transforms for some radial basis functions}
\label{tab::fourier}
\end{table}
In this paper, the second case will be discussed because it is more general.
For the sake of error estimation, some restrictions should be added to the mollifer. We assume that
\begin{equation}\label{eq:13}
 \widehat{\eta}(\xi)\in C_0^\infty~~and~~ \widehat{\eta}(\xi)\equiv 1,~~when~~\|\xi\|_2\leq 1.
\end{equation}
This function is related to the $\sinc$ function. For 1D, one could used
\begin{equation}\label{eq:15}
 \widehat{\eta}(\xi)=
 \begin{cases}
  1,~~~~|\xi|\leq 1\\
  0,~~~~otherwise,
 \end{cases}
\end{equation}
which leads to a $\sinc$ function,
\begin{equation}
 \eta(x)=\dfrac{\sin(x)}{\pi x}.
\end{equation}
For a given $\sigma>0$, set
\begin{equation}\label{eq:14}
 \widehat{\eta}_\sigma (\xi)=\frac{1}{\sigma}\widehat{\eta}(\frac{\xi}{\sigma}),
\end{equation}
its compactly supported interval is $[-\sigma, \sigma]$ and leads to $\eta_\sigma(x)=\frac{\sin(\sigma x)}{\pi\sigma x}$.
\par
For an RBF interpolation function $u(x)=\sum_{j=1}^N\lambda_j\Phi(x-x_j)$, its mollification is defined by
\begin{equation}\label{eq:16}
 u_\sigma(x)=\sum_{j=1}^N\lambda_j\Phi_\sigma(x-x_j)=\sum_{j=1}^N\lambda_j\Phi*\eta_\sigma(x-x_j).
\end{equation}
$u_\sigma(x)$ is a band-limited function with bandwidth $\sigma$, i.e.,
\begin{equation}\label{eq:17}
 supp(\widehat{u_\sigma})\subset [-\sigma,\sigma]
\end{equation}

We now give the corresponding error and stability analysis before discussing the low-rank representation at the heart of the FMM.

\subsection{Error and stability}
A key ingredient of our method is given in the following error bound for bandlimited functions.
\begin{theorem}
\label{thm::error}
Let $\Omega\subseteq \R$ be an open and bounded domain, having a Lipschitz boundary and satisfying the interior cone condition.
Assume that $\Phi(x-x_j)=\varphi(|x-x_j|)$ is a radial basis function such that its generalized Fourier transform exists and satisfies \eqref{eq:12}.
Let $u(x)=\sum_{j=1}^N\lambda_j\Phi(x-x_j)$ be the interpolant on $X=\{x_1,\cdots,x_N\}\subset\Omega$ to a function $f\in W_2^\tau(\Omega)$ and 
$u_\sigma(x)=\sum_{j=1}^N\lambda_j\Phi_\sigma(x-x_j)$ be a band-limited function. Then there exists a positive constant $\kappa=\sigma h$, we have
 \begin{align}\label{eq:19}
  \|f-u_\sigma\|_{W_2^\tau(\Omega)} \leq ch^{s-\tau}\|f\|_{W_2^s(\Omega)}, ~~0\leq \tau\leq s,
  \end{align}
 where, the positive constant $c$ is independent of $h$ and $f$.
\end{theorem}
\begin{proof}
$\Omega$ has a Lipschitz boundary, then there exists an extension mapping $E: W_2^{\tau}(\Omega)\rightarrow W_2^{\tau}(\R)$, 
such that $Ef|_{\Omega}=f$ for $f\in W_2^{\tau}(\Omega)$.
 Moreover, there exists a positive constant $C=C(\tau,\Omega,\R)$ such that
 \begin{equation}\label{eq:19.1}
  \|Ef\|_{W_2^{\tau}(\R)}\leq C\|f\|_{W_2^{\tau}(\Omega)}.
 \end{equation}
 By zero extension, $f$ can be extended from $\Omega$ to $\R$. The extended function is still denoted by $f$ and 
 \begin{equation}
  \|f\|_{W_2^\tau(\R)}\leq C\|f\|_{W_2^\tau(\R)},
 \end{equation}
where $C=C(\tau,\Omega, \R)$.
 For any real integer $\tau$, there exists a positive constant $C$ such that
 \begin{equation}\label{eq:19.2}
  \dfrac{1}{C}\|f\|_{W_2^\tau(\R)}^2\leq \int_{\R}(1+|\xi|^2)^\tau|\widehat{f}(\xi)|^2d\xi\leq C\|f\|_{W_2^\tau(\R)}^2,~~\forall f\in W_2^\tau(\R).
 \end{equation}
 For any given $f\in W_2^\tau(\R)$, define its band-limited function $f_\sigma=f*\eta_\sigma$, then $f_\sigma\in C^\infty(\R)$.
 \begin{align}\label{eq:19.3}
\|f-f_\sigma\|_{W_2^\tau(\Omega)}^2 & \leq \|f-f_\sigma\|_{W_2^\tau(\R)}^2\\
\nonumber & \leq C\int_{\R}(1+|\xi|^2)^\tau|\widehat{(f-f_\sigma})(\xi)|^2d\xi \\
\nonumber & =C\int_{\R}(1+|\xi|^2)^\tau |1-\eta_\sigma(\xi)|^2|\widehat{f}(\xi)|^2d\xi\\
\nonumber & = C\int_{|\xi|>\sigma}(1+|\xi|^2)^s|\widehat{f}(\xi)|^2\dfrac{1}{(1+|\xi|^2)^{s-\tau}}d\xi\\
\nonumber &\leq C\dfrac{1}{(1+\sigma^2)^{s-\tau}}\int_{|\xi|>\sigma}(1+|\xi|^2)^s|\widehat{f}(\xi)|^2d\xi\\
\nonumber &\leq C\dfrac{1}{(1+(\frac{\kappa}{h})^2)^{s-\tau}}\int_{\Bbb{R}}(1+|\xi|^2)^s|\widehat{f}(\xi)|^2d\xi\\
\nonumber & \leq \frac{C}{\kappa^{2(s-\tau)}}h^{2(s-\tau)}\|f\|_{W_2^s(\R)}^2\leq C'h^{2(s-\tau)}\|f\|_{W_2^s(\Omega)}^2.
 \end{align}
According to Plancherel Theorem and \eqref{eq:error}, we have
 \begin{align}\label{eq:19.4}
  \|f_\sigma-u_\sigma\|_{W_2^\tau(\Omega)} & \leq \|f_\sigma-u_\sigma\|_{W_2^\tau(\R)}                                              =\|\widehat{f_\sigma}-\widehat{u_\sigma}\|_{W_2^\tau(\R)}
                                              = \|(\widehat{f}-\widehat{u})\widehat{\eta_\sigma}\|_{W_2^\tau(\R)}\\                                          
\nonumber                                           & \leq (1+\sigma^2)^{\frac{\tau-k}{2}}\|f-u\|_{W_2^k(\R)} \leq ch^{k-\tau}\|f-u\|_{W_2^k(\R)}\\
\nonumber                                           & \leq ch^{k-\tau}\|f-u\|_{W_2^k(\Omega)}
                                             \leq ch^{k-\tau}\cdot ch^{s-k}\|f\|_{W_2^s(\Omega)}
                                             \leq ch^{s-\tau}\|f\|_{W_2^s(\Omega)}.
 \end{align}
Combining \eqref{eq:19.3} with \eqref{eq:19.4}, the following inequality holds:
\begin{equation}\label{eq:19.5}
 \|f-u_\sigma\|_{W_2^\tau(\Omega)} \leq \|f-f_\sigma\|_{W_2^\tau(\Omega)}+\|f_\sigma-u_\sigma\|_{W_2^\tau(\Omega)} \leq ch^{s-\tau}\|f\|_{W_2^s(\Omega)}.                                   
 \end{equation}
\end{proof}
Using Theorem \ref{thm::error}, the mollification $u_\sigma$ can be used as an approximation to the function $f\in W_2^\tau(\Omega)$. 
A standard criterion for measuring the numerical stability of an approximation method is its condition number.
We need to consider the condition number of the interpolation matrix $A$ with entries $A_{ij}=\Phi_\sigma(x_i-x_j)$. 
If $A$ is positive definite, then its $l_2$- condition number is given by
\begin{equation}
 cond(A)=\|A\|_2\|A^{-1}\|_2=\dfrac{\gamma_{\max}(A)}{\gamma_{\min}(A)},
\end{equation}
where $\gamma_{\max}$ is the maximum eigenvalue and $\gamma_{\min}$ is the minimum eigenvalue.
\par
From Gershgorin's theorem, it is easy to obtain 
\begin{equation}
 \gamma_{\max}\leq N\max_{j,k=1,\cdots,N}|\Phi_\sigma(x_j-x_k)|.
\end{equation}
Because $X$ in this paper is quasi-uniformly distributed, in fact, as long as its variation not too wildly, $N$ will grow as $h^{-1}$ which makes the growth of $\gamma_{\max}$ acceptable and 
hence
\begin{equation}
 \gamma_{\max}\leq Ch^{-1}.
\end{equation}
We focus on finding lower bounds for the minimum eigenvalue.
\begin{theorem}
Let $q_{X}:=\frac{1}{2}\min_{j\neq k}\|x_j-x_k\|_2$ be the separation distance of the set $X$ and let $\Phi$ be a radial basis function and 
$\Phi_\sigma=\Phi*\eta_\sigma$. For the  interpolation matrix with entries $\Phi_\sigma(x_j-x_k)$, we have
 \begin{equation}\label{eq:constrain}
  \sum_{j,k=1}^N\lambda_j\lambda_k\Phi_\sigma(x_j-x_k)\geq \gamma_{\min}\|\lambda\|_2^2,
 \end{equation}
 with $\gamma_{\min}=q_{X}^{-1}\widehat{\Phi}(\frac{2\pi}{q_{X}})$.
\end{theorem}
\begin{proof}
We start with
\begin{align}\label{eq:minimumeig}
 \sum_{j,k=1}^N\lambda_j\lambda_k\Phi_{\sigma}(x_j-x_k)
                            & =\sum_{j,k=1}^N\lambda_j\lambda_k\frac{1}{2\pi}\int_{-\sigma}^{\sigma}\widehat{\Phi}(\xi)e^{i\xi(x_j-x_k)}d\xi\\
\nonumber                            & \geq \sum_{j,k=1}^N\lambda_j\lambda_k\frac{1}{2\pi}\int_{-\sigma}^{\sigma}(1-\frac{|\xi|}{\sigma})\widehat{\Phi}(\xi)e^{i\xi(x_j-x_k)}d\xi\\
\nonumber                            & \geq [\frac{1}{2\pi}\inf_{\xi\in[-\sigma,\sigma]}\widehat{\Phi}(\xi)]\sum_{j,k=1}^N\lambda_j\lambda_k\int_{-\sigma}^{\sigma}(1-\frac{|\xi|}{\sigma})e^{i\xi(x_j-x_k)}d\xi\\
\nonumber                            & =\underbrace{[\frac{\sigma}{2\pi}\inf_{\xi\in[-\sigma,\sigma]}\widehat{\Phi}(\xi)]}_{\text{Part I}}\underbrace{\sum_{j,k=1}^N\lambda_j\lambda_k \sinc^2(\frac{\sigma}{2}(x_j-x_k))}_{\text{Part II}}.
\end{align}
For Part I, $\widehat{\Phi}(\xi)$ is clearly decreasing. Thus the infimum takes the value
\begin{equation}
\frac{\sigma}{2\pi}\inf_{\xi\in[-\sigma,\sigma]}\widehat{\Phi}(\xi)=\frac{\sigma}{2\pi}\widehat{\Phi}(\sigma).
\end{equation}
For Part II, we use
\begin{align}
  \sum_{j,k=1}^N\lambda_j\lambda_k \sinc^2(\frac{\sigma}{2}(x_j-x_k)) & \geq \|\lambda\|_2^2 \sinc^2(0)-\sum_{j\neq k}|\lambda_j\|\lambda_k|\sinc^2(\frac{\sigma}{2}(x_j-x_k))\\
                                                          \nonumber           & \geq \|\lambda\|_2^2 \sinc^2(0)-\frac{1}{2}\sum_{j\neq k}(|\lambda_j|^2+|\lambda_k|^2)\sinc^2(\frac{\sigma}{2}(x_j-x_k))\\                                                   \nonumber           &  =   \|\lambda\|_2^2(1-\max_{1\leq j\leq N}\sum_{k=1,k\neq j}^N \sinc^2(\frac{\sigma}{2}(x_j-x_k))).
 \end{align}
For the chosen $\sigma$, let
\begin{equation}
 \max_{1\leq j\leq N}\sum_{k=1,k\neq j}^N \sinc^2(\frac{\sigma}{2}(x_j-x_k))= \frac{1}{2}.
\end{equation}
Assume that the maximum is taken for $x_1=0$, i.e. that
\begin{equation}
 \max_{1\leq j\leq N}\sum_{k=1,k\neq j}^N \sinc^2(\frac{\sigma}{2}(x_j-x_k))=\sum_{k=2}^N\sinc^2(\frac{\sigma}{2}x_k).
\end{equation}
Every $x_j(2\leq j\leq N)$ is contained in 
\begin{equation}
 \mathcal{E}_n=\{x\in\Bbb{R}:nq_{X}\leq |x| < (n+1)q_{X},n\geq 1\}.
\end{equation}
Every ball $B(x_j,q_{X})$ around $x_j$ with radius $q_{X}$ is disjoint from a ball around $x_k (k\neq j)$ with the same radius and these balls are 
contained in 
\begin{equation}
 \{x\in\Bbb{R}:(n-1)q_{X}\leq |x| \leq (n+2)q_{X}\}.
\end{equation}
The number of points in $\mathcal{E}_n (n\geq 1)$ can be get by computing volumes
\begin{equation}
 \#\{x_j\in\mathcal{E}_n\}\leq (n+2)-(n-1) \leq 3.
\end{equation}
Thus, if we use  $\sum_{n=1}^{\infty}\frac{1}{n^2}=\frac{\pi^2}{6}$, we have
\begin{align}
 \sum_{k=2}^N\sinc^2(\frac{\sigma}{2}x_k) &\leq \sum_{n=1}^{\infty}\#\{x_j\in\mathcal{E}_n\}\sup_{x\in\mathcal{E}_n}\sinc^2(\frac{\sigma}{2}x)  \leq \sum_{n=1}^{\infty}\#\{x_j\in\mathcal{E}_n\}\sup_{x\in\mathcal{E}_n}\dfrac{1}{(\frac{\sigma}{2}x)^2}\\
\nonumber                                           &\leq \sum_{n=1}^{\infty}3\dfrac{1}{(\frac{\sigma}{2}nq_{X})^2}
                                           =\frac{2\pi^2}{(\sigma q_{X})^2}.
\end{align}
When $\sigma=\frac{2\pi}{q_{X}}$, \eqref{eq:constrain} holds, then
\begin{equation}\label{eq:mine}
 \sum_{j,k=1}^N\lambda_j\lambda_k\Phi_{\sigma}(x_j-x_k) \geq \frac{\sigma}{2\pi}\widehat{\Phi}(\sigma)\frac{1}{2}||\mathbf{\lambda}||_2^2
                           \geq q_{X}^{-1}\widehat{\Phi}(\frac{2\pi}{q_{X}})||\mathbf{\lambda}||_2^2.
\end{equation}
\end{proof}	
We then obtain the following result for the condition number of the interpolation matrix $A$.
\begin{Corollary}
 Assume that $\Phi$ is a radial basis function such that its generalized Fourier transform exists and satisfies \eqref{eq:12}. Then
 \begin{equation}
  \mathrm{cond}(A)\leq cq_{X}^{-2\tau}.
 \end{equation}
\end{Corollary}
Our next result discloses the connection of the error between $u$ and $u_\sigma$.
\begin{theorem}
Let $\Omega\subseteq \R$ be an open and bounded domain, having a Lipschitz boundary and satisfying the interior cone condition.
Assume that $\Phi(x-x_j)=\varphi(|x-x_j|)$ is a radial basis function such that its generalized Fourier transform exists and satisfies \eqref{eq:12}.
Let $u(x)=\sum_{j=1}^N\lambda_j\Phi(x-x_j)$ be the interpolant on $X=\{x_1,\cdots,x_N\}\subset\Omega$ to a function $f\in W_2^\tau(\Omega)$ and 
$u_\sigma(x)=\sum_{j=1}^N\lambda_j\Phi_\sigma(x-x_j)$ be a band-limited function. Then there exists a positive constant such that
\begin{equation}
\|u-u_\sigma\|_{W_2^\tau(\Omega)}\leq c\|u\|_{W_2^\tau (\Omega)}
\end{equation}
\end{theorem}
\begin{proof} Since $\Omega$ has a Lipschitz boundary, there still exists a zero extension.
\par
From a change of variable $\xi=\sigma\omega$, we obtain that
\begin{equation}
 \|u-u_\sigma\|_{W_2^\tau(\R)}^2 = \sum_{j,k=1}^N\lambda_j\lambda_k\sigma\int_{\|\omega\|_2\geq1}\frac{1}{(1+\|\sigma\omega\|_2^2)^\tau}e^{i\sigma\omega(x_j-x_k)}d\omega
\end{equation}
Since $\|\omega\|\geq 1$, we have $\dfrac{1}{(1+\sigma^2\|\omega\|_2^2)^\tau}\leq \dfrac{2^\tau}{\sigma^{2\tau}}\dfrac{1}{(1+\|\omega\|_2^2)^\tau}$~\cite{rieger2008sampling}, so that
\begin{align}
 \|u-u_\sigma\|_{W_2^\tau(\R)}^2 & \leq 2^\tau\sigma^{1-2\tau}\sum_{j,k=1}^N\lambda_j\lambda_k\int_{\|\omega\|_2\geq1}(1+\|\omega\|_2^2)^{-\tau}e^{i\sigma\omega(x_j-x_k)}d\omega\\
\nonumber                                    & \leq 2^\tau\sigma^{1-2\tau}\sum_{j,k=1}^N\lambda_j\lambda_k\int_{\R}(1+\|\omega\|_2^2)^{-\tau}e^{i\sigma\omega(x_j-x_k)}d\omega\\
\nonumber                                    & =2^\tau\sigma^{1-2\tau}\sum_{j,k=1}^N\lambda_j\lambda_k\int_{\R}\widehat{\Phi}(\omega)e^{i\sigma\omega(x_j-x_k)}d\omega\\
\nonumber                                    & =2^\tau\sigma^{1-2\tau}2\pi\sum_{j,k=1}^N\lambda_j\lambda_k\Phi(\sigma(x_j-x_k))\\
\nonumber                                    & =2^\tau \kappa^{1-2\tau}2\pi h^{2\tau-1}\sum_{j,k=1}^N\lambda_j\lambda_k\Phi(\sigma(x_j-x_k))\\
\nonumber                                    & =2^\tau \kappa^{1-2\tau}2\pi h^{2\tau-1}(\sigma h)^{-1} \|\mathbf{\lambda}\|_2^2\\
\nonumber                                    & =2^\tau \kappa^{-2\tau}2\pi h^{2\tau-1}\|\mathbf{\lambda}\|_2^2\\
\nonumber                                    & \leq ch^{2\tau-1}\|\mathbf{\lambda}\|_2^2.
 \end{align}
From \eqref{eq:constrain}, it becomes
\begin{align}
 \|u-u_\sigma\|_{W_2^\tau(\Omega)}^2 &\leq \|u-u_\sigma\|_{W_2^\tau(\R)}^2 \leq ch^{2\tau-1}\frac{1}{\gamma_{\min}}\sum_{j,k=1}^N\lambda_j\lambda_k\Phi_\sigma(x_j-x_k)\\
\nonumber                                      & \leq ch^{2\tau-1}\frac{1}{\gamma_{\min}}\sum_{j,k=1}^N\lambda_j\lambda_k\Phi(x_j-x_k)\\
\nonumber                                      & =ch^{2\tau-1}\frac{1}{\gamma_{\min}} (\sum_{j=1}^N\lambda_j\Phi(x-x_j), \sum_{k=1}^N\lambda_k\Phi(x-x_k))_{\mathcal{N}_{\Phi}}\\
\nonumber                                      & =ch^{2\tau-1}\frac{1}{\gamma_{\min}} \|u\|_{W_2^\tau(\R)}^2\leq C\|u\|_{W_2^\tau (\Omega)}^2.
 \end{align}
\end{proof}
For our numerical illustrations we choose the following simple boundary value problem:
\begin{equation}\label{eq:1Dmodel}
\begin{cases}
 -\Delta u(x)+\pi^2u(x)=2\pi^2\sin\pi x, x\in(0,\pi),\\
 u(0)=u(\pi)=0,
\end{cases}
\end{equation}
which has exact solution $u(x)=\sin\pi x$.
We now compare the band-limited approximation to the unsymmetric collocation based on $MQ$ functions. In the left half of Table \ref{tab::RMS} we apply $\Phi(r)=\sqrt{r^2+1}$ and in the right half we use $\Phi*\eta_\sigma(r)$ to evaluate the root-mean-square error (RMS-error). 
\begin{table}[H]
\centering
\begin{tabular}{ccc}
\toprule
 $N$          &$\Phi(r)$                                &$\Phi*\eta_\sigma(r)$\\
 \midrule
 9            &1.469348643e-04                           &1.469348658e-04\\
10            &9.414500417e-05                           &9.414500776e-05 \\
11            &2.806645307e-05                           &2.806731328e-05\\
12            &1.823679202e-05                           &1.823613930e-05\\
13            &5.348123608e-06                           &5.345923089e-06\\
14            &3.512156051e-06                           &3.512046451e-06\\ 
15            &1.007928224e-06                           &7.282274291e-07\\
\bottomrule
\end{tabular}
\caption{RMS errors for the approximate solution.}
\label{tab::RMS}
\end{table}
Having discussed the approximation quality of the band-limited approximation we now come to the multilevel fast multipole method.
\subsection{Low-rank representation}
The inverse Fourier transform for a given $\Phi_\sigma$ is expressed as
\begin{equation}
 \Phi_\sigma(x-x_j)=\dfrac{1}{2\pi}\int_{-\sigma}^{\sigma}\widehat{\Phi}(\xi)e^{i\xi(x-x_j)}d\xi.
\end{equation}
It can be approximated by constructing a simple numerical quadrature to obtain
\begin{equation}\label{eq:quadrature}
 \Phi_\sigma(x-x_j)=\sum_{m=1}^M\omega_m\widehat{\Phi}(\xi_m)e^{i\xi_m(x-x_j)}+\varepsilon_M,
\end{equation}
with quadrature weights $\omega_m$ and error term $\varepsilon_M$.
Next, we use a Fourier series form of the term $\widehat{\Phi}(\xi_m)$ given by
\begin{equation}
 \widehat{\Phi}(\xi_m)\approx\sum_{q=-Q}^Q\Phi(q)e^{-iq\xi_m}.
\end{equation}
Then the expansion for $\Phi_\sigma(x-x_j)$ is given by
\begin{equation}\label{eq:fmmtype}
 \Phi_\sigma(x-x_j)=\sum_{m=1}^M\mathcal{U}(\xi_m)\mathcal{C}(\xi_m)\mathcal{V}(\xi_m)+\varepsilon,
\end{equation}
where $\mathcal{C}(\xi_m)$ is the translation operator given by
\begin{equation}\label{eq:M2Loperator}
 \mathcal{C}(\xi_m)=\dfrac{1}{2\pi}\sum_{q=-Q}^{Q}\Phi(q)e^{-iq\xi_m},
\end{equation}
where
$\mathcal{V}(\xi_m)$ is the multipole expansion (aggregation) of the source points given by
\begin{equation}\label{eq:P2M}
 \mathcal{V}(\xi_m)=e^{-i\xi_mx_j},
\end{equation}
and $\mathcal{U}(\xi_m)=\omega_me^{i\xi_m x}$ is the 
L2P operator. In practice, one typically chooses
 \begin{align*}
  \omega_m\rightarrow \Delta\xi=2\sigma/M,\\
  \xi_m=-\sigma+(m-1)\Delta\xi,\\
  Q=M/2~\textrm{ and }~\sigma=\pi.
 \end{align*}
The formula \eqref{eq:fmmtype} provides the starting point for the FMM in this paper. This construction can be extened to the 2D or higher dimensions. 
For example, in 2D, the band-limited approximation reads:
\begin{align} \label{eq:28}
\Phi_\sigma(x-x_j,y-y_j)&=\dfrac{1}{(2\pi)^2}\int_{-\sigma}^\sigma\int_{-\sigma}^\sigma\widehat{\Phi}(\xi_1,\xi_2)e^{i\xi_1(x-x_j)}e^{i\xi_2(y-y_j)}d\xi_1 d\xi_2\\
\nonumber                        &=\dfrac{1}{(2\pi)^2}\sum_{m_1,m_2}\omega_{1_{m_1}}\omega_{2_{m_2}}\widehat{\Phi}(\xi_{1_{m_1}},\xi_{2_{m_2}})e^{i\xi_{1_{m_1}}(x-x_j)}e^{i\xi_{2_{m_2}}(y-y_j)}+\varepsilon_{\mathbf{M}}\\
 \nonumber                        &=\sum_{m_1,m_2}\omega_{1_{m_1}}\omega_{2_{m_2}}\mathcal{C}(\xi_{1_{m_1}},\xi_{2_{m_2}})e^{i\xi_{1_{m_1}}(x-x_j)}e^{i\xi_{2_{m_2}}(y-y_j)}+\varepsilon_{\mathbf{M}},
 \end{align}
with $\mathcal{C}(\xi_{1_{m_1}},\xi_{2_{m_2}})=\dfrac{1}{(2\pi)^2}\widehat{\Phi}(\xi_{1_{m_1}},\xi_{2_{m_2}})$. Let $\mathbf{x}=(x,y),~\mathbf{x}_j=(x_j,y_j)$ and $\mathbf{\xi}=(\xi_1,\xi_2),~\mathbf{\omega}=(\omega_1,\omega_2)$, \eqref{eq:28} can be rewritten as
\begin{equation}
 \label{eq:29} \Phi_\sigma(\mathbf{x}-\mathbf{x}_j)=\sum_{\mathbf{m}=1}^{\mathbf{M}}\mathbf{\omega}_{\mathbf{m}}\mathcal{C}(\mathbf{\xi}_\mathbf{m})e^{i\mathbf{\xi}_{\mathbf{m}}(\mathbf{x}-\mathbf{x}_j)}+\varepsilon_{\mathbf{M}},
\end{equation}
with $\mathcal{C}(\mathbf{\xi}_\mathbf{m})=\dfrac{1}{(2\pi)^2}\widehat{\Phi}(\mathbf{\xi}_{\mathbf{m}})$.
We have not yet exploited the band-limited approximation for the usage within the FMM and will do this in the following.
\section{Fast Multipole Method (FMM) Based on Bandlimited Function for RBFs}\label{subsection::FMMRBF}
In this section, we will discuss the FMM in 2D. We neglect the error and approximate $\Phi_\sigma$ by
\begin{equation}\label{eq:35}
 \Phi_\sigma^{FMM}(\mathbf{x}_i-\mathbf{x}_j)=\sum_{\mathbf{m}=1}^{\mathbf{M}}\mathbf{\omega}_{\mathbf{m}}\mathcal{C}(\mathbf{\xi}_{\mathbf{m}})e^{i\mathbf{\xi}_{\mathbf{m}}(\mathbf{x}_i-\mathbf{x}_j)},
\end{equation}
where $\mathcal{C}(\mathbf{\xi}_{\mathbf{m}})=\dfrac{1}{(2\pi)^2}\widehat{\Phi}(\mathbf{\xi}_{\mathbf{m}})$. 
\par
Figure \ref{fig::FMM} graphically illustrates the construction within the FMM. Let $\mathbf{x}_i$ and $\mathbf{x}_j$ be the evaluation point and source point, respectively. 
For two well-separated squares $a$ and $b$, $\mathbf{x}_a$ and $\mathbf{x}_b$ are their centers and $\mathbf{x}_i\in a$ and $\mathbf{x}_j\in b$.
We have:
\begin{equation}\label{eq:36}
 \mathbf{x}_i-\mathbf{x}_j=(\mathbf{x}_i-\mathbf{x}_a)+(\mathbf{x}_a-\mathbf{x}_b)+(\mathbf{x}_b-\mathbf{x}_j)=\mathbf{r}_{ia}+\mathbf{r}_{ab}+\mathbf{r}_{bj}.
\end{equation}
\begin{figure}[H]
\includegraphics[height=2.4in,width=5in]{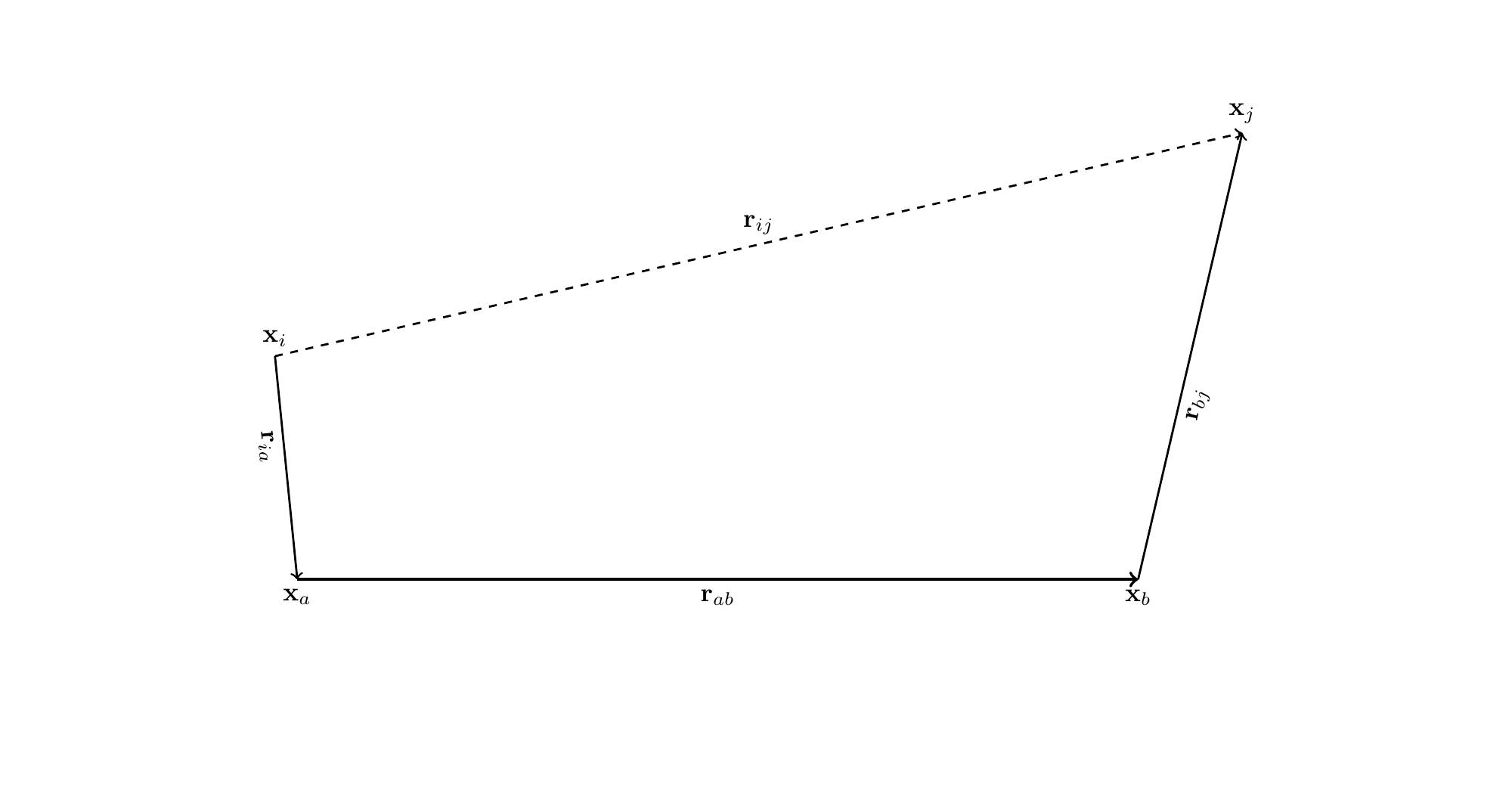}
\caption{\label{fig::FMM}Vector definitions for FMM expansion.}
\end{figure}
It is easy to see that \eqref{eq:35} can be rewritten as a low-rank approximation,
\begin{equation}\label{eq:37}
 \Phi_\sigma^{FMM}(\mathbf{x}_i-\mathbf{x}_j)
 =\sum_{\mathbf{m}=1}^{\mathbf{M}}\mathbf{\omega}_{\mathbf{m}}e^{i\mathbf{\xi}_{\mathbf{m}}\mathbf{r}_{ia}}\mathcal{C}(\mathbf{\xi}_{\mathbf{m}})e^{i\mathbf{\xi}_{\mathbf{m}}\mathbf{r}_{ab}}e^{i\mathbf{\xi}_{\mathbf{m}}\mathbf{r}_{bj}}.
\end{equation}
When we use an iterative method to solve \eqref{eq:1}, the necessary matrix-vector multiply hast to be computed in each iteration. This typically represents the bottleneck of any iterative solver. We can now express the matrix-vector multiply as
\begin{align}\label{eq:38} 
 &\sum_{j=1}^N\lambda_j\Phi(\mathbf{x}_i-\mathbf{x}_j)\\
\nonumber &\approx\underbrace{\sum_{b\in \mathcal{N}_a}\sum_{j\in \mathcal{G}_b}     \lambda_j\Phi(\mathbf{x}_i-\mathbf{x}_j)}_{\text{Near-field}}+
 \underbrace{\underbrace{\sum_{\mathbf{m}=1}^{\mathbf{M}}e^{i\mathbf{\xi}_{\mathbf{m}}\mathbf{r}_{ia}}}_{\text{disaggregation}}\underbrace{\sum_{b\notin\mathcal{N}_a}\mathbf{\omega}_{\mathbf{m}}\mathcal{C}
 (\mathbf{\xi}_{\mathbf{m}})e^{i\mathbf{\xi}_{\mathbf{m}}\mathbf{r}_{ab}}}_{\text{translation}}\underbrace{\sum_{j\in\mathcal{G}_b}\lambda_je^{i\mathbf{\xi}_{\mathbf{m}}\mathbf{r}_{bj}}}_{\text{aggregation}}}_{\text{Far-field}}	
\end{align}
where $\mathcal{G}_a$ denotes all particles in group $a$ and $\mathcal{N}_a$ denotes all neighbour groups of group $a$, $1\leq a,b\leq p$ and $p=N/\overline{N}$ is the number of groups.
The total complexity of the FMM is estimated as follows:\\
1. Near-field : $T_1=C_1ep\overline{N}^2=C_1eN\overline{N}$, $e$ is the average number of neighbors and $C_1$ is a constant.\\
2. Aggregation: $T_2=C_1M\overline{N}p=C_1MN$.\\
3. Translation: $T_3=C_1Mp(p-e)$.\\
4. Disaggregation: $T_4=C_1M\overline{N}p=C_1MN$.\\
The total complexity: 
\begin{equation}
T=C_1eN\overline{N}+C_1MN+C_1Mp(p-e)+C_1MN.
\end{equation}
Minimizing with respect to $M$ yields the result of $\mathcal{O}(N^{1.5})$ for $M\sim \overline{N} \sim \sqrt{N}$.

A further reduction in the computational cost is achieved when the FMM is replaced by a multilevel approximation.
\subsection{Multilevel Fast Multipole Method (MLFMM)}
The idea behind FMM was extended and applied in a recursive manner, leading to the multilevel fast multipole method (MLFMM) (see for example \cite{MarR05}). 
This algorithm has three steps: the upsweep corresponds to building and propagating multipole expansions (M2P and M2M) up the tree, the coupling phase corresponds to computing the M2L operator, and the downsweep corresponds to propagating local expansions (L2L and L2P). Before starting the algorithm, it is necessary to  recall the Weierstrass approximation theorem.
\begin{theorem}(Weierstrass Approximation Theorem)
Let $g$ be a continuous function on the closed and bounded interval $[a,b]\subset\R$. Then, for any $\varepsilon>0$, there exists a polynomial $P$ such that
 \begin{equation}
  \sup_{x\in[a,b]}|g(x)-P(x)|<\varepsilon.
 \end{equation}
In other words, any continuous function on a closed and bounded interval can be uniformly approximated on that interval by polynomials to any degree of accuracy.
\end{theorem}
We now discuss the steps of the MLFMM in more detail.\\
\textbf{Upsweep}: The multipole expansions are computed at the finest level, and then the expansions for the coarser level are obtained using interpolation and shifting.
Let $\mathbf{x}_{b_l}$ and $\mathbf{x}_{b_{l-1}}$ be centers of square $b$ at level $l$ and $l-1$, respectively. At the finest level $l$, the multipole expansion 
$\mathcal{V}_{b_l}(\mathbf{\xi}_{l_{\mathbf{m}_{(l)}}})=e^{i\mathbf{\xi}_{l_{\mathbf{m}_{(l)}}}\mathbf{r}_{b_lj}} (\mathbf{m}_{(l)}=1,2,\cdots,\mathbf{M}_{(l)}$) has $\mathbf{M}_{(l)}$ values. For level $l-1$, we need $\mathbf{M}_{(l-1)}$ values of $\mathcal{V}_{b_{l-1}}(\mathbf{\xi}_{{l-1}_{\mathbf{m}_{(l-1)}}})$. According to the Weierstrass approximation theorem, we can use 
a polynomial interpolation method to obtain the $\mathbf{M}_{(l-1)}$ values. Then, the multipole expansion for level $l-1$ will be
\begin{equation}\label{eq:43}
 \mathcal{V}_{b_{{(l-1)}}}(\mathbf{\xi}_{{l-1}_{\mathbf{m}_{(l-1)}}})=\underbrace{\underbrace{e^{i\mathbf{\xi}_{{l-1}_{\mathbf{m}_{(l-1)}}}(\mathbf{x}_{b_{l-1}}-\mathbf{x}_{b_{l}})}}_{\text{translation operator}}
 \sum_{\mathbf{m}_{(l)}=1}^{\mathbf{M}_{(l)}}\mathcal{P}_{\mathbf{m}_{(l)}}(\mathbf{\xi}_{{l-1}_{\mathbf{m}_{(l-1)}}})\underbrace{\mathcal{V}_{b_l}(\mathbf{\xi}_{l_{\mathbf{m}_{(l)}}})}_{\text{P2M operator at level $l$}}}_{\text{M2M operator from level $l$ to level $l-1$}},
\end{equation}
where $\mathbf{m}_{(l-1)}=1,2,\cdots,\mathbf{M}_{(l-1)}$ and $\mathcal{P}_{\mathbf{m}_l}(\mathbf{\xi}_{{l-1}_{\mathbf{m}_{(l-1)}}})$ are the interpolation coefficients. This process will stop at level 2. At this level, the multipole expansion is expressed as
\begin{equation}\label{eq:44}
 \mathcal{V}_{b_{2}}(\mathbf{\xi}_{2_{\mathbf{m}_{(2)}}})=e^{i\mathbf{\xi}_{{2}_{\mathbf{m}_{(2)}}}(\mathbf{x}_{b_{2}}-\mathbf{x}_{b_{3}})}\sum_{\mathbf{m}_{(3)}=1}^{\mathbf{M}_{(3)}}\mathcal{P}_{\mathbf{m}_{(3)}}(\mathbf{\xi}_{2_{\mathbf{m}_{(2)}}})\mathcal{V}_{b_3}(\mathbf{\xi}_{3_{\mathbf{m}_{(3)}}}),~~~\mathbf{m}_{(2)}=1,2,\cdots,\mathbf{M}_{(2)}.
\end{equation}
\textbf{Coupling.} The translation of the multipole expansion to the local expansion is completed by a multiplication with
\begin{equation}
\mathcal{T}_{(a_2b_2)}=\mathbf{\omega}_{2_{\mathbf{m}_{(2)}}}\mathcal{C}_{a_2}(\mathbf{\xi}_{2_{\mathbf{m}_{(2)}}})e^{i\mathbf{\xi}_{2_{\mathbf{m}_{(2)}}}(\mathbf{x}_{a_2}-\mathbf{x}_{b_2})}.
\end{equation}
Then the local expansion is of the form
\begin{equation}\label{eq:45}
 \mathcal{L}_{a_{2}}(\mathbf{\xi}_{2_{\mathbf{m}_{(2)}}})=\mathcal{C}_{a_2}(\mathbf{\xi}_{2_{\mathbf{m}_{(2)}}})e^{i\mathbf{\xi}_{2_{\mathbf{m}_{(2)}}}(\mathbf{x}_{a_2}-\mathbf{x}_{b_2})}\mathcal{V}_{b_{2_{\mathbf{m}_{(2)}}}}(\mathbf{\xi}_{2_{\mathbf{m}_{(2)}}}).
\end{equation}
\textbf{Downsweep.} We need to scatter local expansion down to the leaves. It is the inverse process of aggregation.
If the local expansions received at level $l-1$ are $\mathcal{L}_{a_{{(l-1)}}}(\mathbf{\xi}_{{l-1}_{\mathbf{m}_{(l-1)}}})$, then the contribution from all well-separated groups can be expressed as
\begin{equation}\label{eq:46}
 \mathcal{I}_{a_{l-1}}=\sum_{\mathbf{m}_{(l-1)}=1}^{\mathbf{M}_{(l-1)}}\mathbf{\omega}_{{l-1}_{\mathbf{m}_{(l-1)}}}e^{i\mathbf{\xi}_{{l-1}_{\mathbf{m}_{(l-1)}}}(\mathbf{x}_i-\mathbf{x}_{a_{l-1}})}\mathcal{L}_{a_{{l-1}}}(\mathbf{\xi}_{{l-1}_{\mathbf{m}_{(l-1)}}}).
\end{equation}
Afterwards, substituting the interpolation expression for $e^{i\mathbf{\xi}_{{l-1}_{\mathbf{m}_{(l-1)}}}(\mathbf{x}_i-\mathbf{x}_{a_{l-1}})}$, and changing the order of the two summations leads to
\begin{equation}\label{eq:47}
 e^{i\mathbf{\xi}_{{l-1}_{\mathbf{m}_{(l-1)}}}(\mathbf{x}_i-\mathbf{x}_{a_{l-1}})}=e^{i\mathbf{\xi}_{l_{\mathbf{m}_{(l)}}}(\mathbf{x}_{a_l}-\mathbf{x}_{a_{l-1}})}\sum_{\mathbf{m}_{(l)}=1}^{\mathbf{M}_{(l)}}e^{i\mathbf{\xi}_{l_{\mathbf{m}_{(l)}}}(x_i-x_{a_{l}})}\mathcal{P}_{\mathbf{m}_{(l)}}^T(\mathbf{\xi}_{{l-1}_{\mathbf{m}_{(l-1)}}}),
\end{equation}
which we then use to get
\begin{align}\label{eq:48}
\mathcal{I}_{a_l}&=\underbrace{\sum_{\mathbf{m}_{(l)}=1}^{\mathbf{M}_{(l)}}\mathbf{\omega}_{l_{\mathbf{m}_{(l)}}} e^{i\mathbf{\xi}_{l_{\mathbf{m}_{(l)}}}(\mathbf{x}_i-\mathbf{x}_{a_{l}})}}_{\text{L2P operator at level $l$:~$\mathcal{U}_{a_l}$}}\\
\nonumber &\underbrace{\sum_{\mathbf{m}_{(l-1)}=1}^{\mathbf{M}_{(l-1)}}\mathbf{\omega}_{{l-1}_{\mathbf{m}_{(l-1)}}}/\mathbf{\omega}_{l_{\mathbf{m}_{(l)}}}\mathcal{P}_{\mathbf{m}_{(l)}}^T(\mathbf{\xi}_{{l-1}_{\mathbf{m}_{(l-1)}}})
 e^{i\mathbf{\xi}_{l_{\mathbf{m}_{(l)}}}(\mathbf{x}_{a_l}-\mathbf{x}_{a_{l-1}})}}_{\text{L2L operator from level $l-1$ to level $l$}}\mathcal{L}_{a_{{l-1}_{\mathbf{m}_{(l-1)}}}}.
\end{align}
To clarify the process of the MLFMM, we look into a simple three-level formulation for the indirect evaluation of the interaction between a source point and an evaluation point. 
Consider the three-level vector construct shown in Figure \ref{fig::MLFMM}, where a simple case is shown to understand the discussed properties better.
\begin{figure}[H]
\centering
\includegraphics[height=2.2in,width=5in]{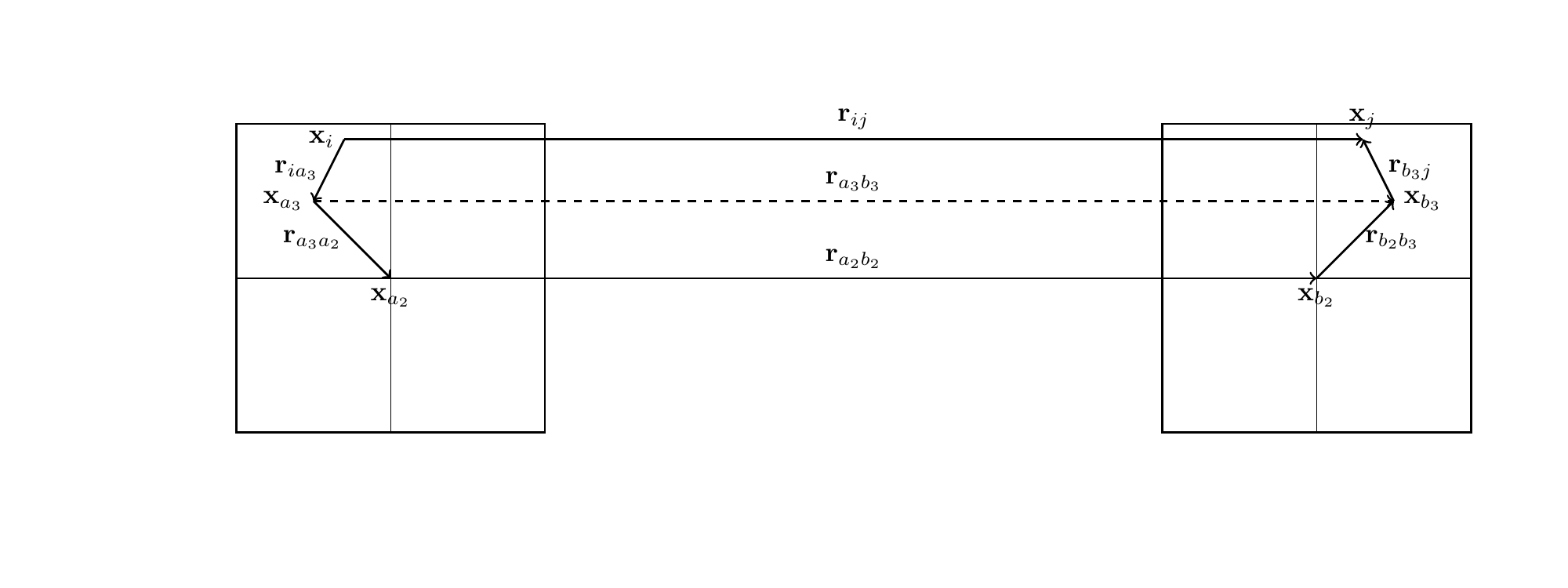}
\caption{\label{fig::MLFMM}Vector definitions for three-level FMM expansion.}
\end{figure}
\begin{equation}\label{eq:49}
 \mathbf{x}_i-\mathbf{x}_j=\mathbf{r}_{ij}=\mathbf{r}_{ia_{3}}+\mathbf{r}_{a_3a_2}+\mathbf{r}_{a_2b_2}+\mathbf{r}_{b_2b_3}+\mathbf{r}_{b_3j}.
\end{equation}
Because the multipole-to-local (M2L) translation will take place at level 2, discretizing the function $\Phi_\sigma(\mathbf{x}_i-\mathbf{x}_j)$ using numerical quadrature for level 2, we obtain
\begin{align}
\label{eq:50} 
&\Phi_\sigma^{FMM}(\mathbf{x}_i-\mathbf{x}_j)= \sum_{\mathbf{m}_{(2)}=1}^{\mathbf{M}_{(2)}}\omega_{2_{\mathbf{m}_{(2)}}}\mathcal{C}(\mathbf{\xi}_{2_{\mathbf{m}_{(2)}}})e^{i\mathbf{\xi}_{2_{\mathbf{m}_{(2)}}}(\mathbf{r}_{ia_{3}}+\mathbf{r}_{a_3a_2}+\mathbf{r}_{a_2b_2}+\mathbf{r}_{b_2b_3}+\mathbf{r}_{b_3j})}\\
\nonumber&=\sum_{\mathbf{m}_{(2)}=1}^{\mathbf{M}_{(2)}}\omega_{2_{\mathbf{m}_{(2)}}}e^{i\mathbf{\xi}_{2_{\mathbf{m}_{(2)}}}\mathbf{r}_{ia_{2}}}e^{i\mathbf{\xi}_{2_{\mathbf{m}_{(2)}}}\mathbf{r}_{a_3a_2}}\mathcal{C}\nonumber(\mathbf{\xi}_{3_{\mathbf{m}_{(2)}}})e^{i\mathbf{\xi}_{2_{\mathbf{m}_{(2)}}}\mathbf{r}_{a_2b_2}}e^{i\mathbf{\xi}_{2_{\mathbf{m}_{(2)}}}\mathbf{r}_{b_2b_3}}e^{i\mathbf{\xi}_{2_{\mathbf{m}_{(2)}}}\mathbf{r}_{b_3j}}
\end{align}
It is observed that the discrete values of $e^{i\mathbf{\xi}_{2_{\mathbf{m}_{(2)}}}\mathbf{r}_{ia_{3}}}$ and $e^{i\mathbf{\xi}_{2_{\mathbf{m}_{(2)}}}\mathbf{r}_{b_3j}}$ are for the frequency discretization at level 2. We do not compute their values directly. In order to save computing time, we obtain their values indirectly from the finer level 3 via an interpolation method. From the known values $e^{i\mathbf{\xi}_{3_{\mathbf{m}_{(3)}}}\mathbf{r}_{b_3j}}$ and $e^{i\mathbf{\xi}_{3_{\mathbf{m}_{(3)}}}\mathbf{r}_{ia_3}}$ at level 3, the approximations are of the form
\begin{align}
\label{eq:51}
 &e^{i\mathbf{\xi}_{2_{\mathbf{m}_{(2)}}}\mathbf{r}_{b_3j}}=\sum_{\mathbf{m}_{(3)}=1}^{\mathbf{M}_{(3)}}\mathcal{P}_{\mathbf{m}_{(3)}}(\mathbf{\xi}_{2_{\mathbf{m}_{(2)}}})e^{i\mathbf{\xi}_{3_{\mathbf{m}_{(3)}}}\mathbf{r}_{b_3j}},\\
\nonumber &e^{i\mathbf{\xi}_{2_{\mathbf{m}_{(2)}}}\mathbf{r}_{ia_{3}}}=\sum_{\mathbf{m}_{(3)}=1}^{\mathbf{M}_{(3)}}e^{i\mathbf{\xi}_{3_{\mathbf{m}_{(3)}}}\mathbf{r}_{ia_{3}}}\mathcal{P}_{\mathbf{m}_{(3)}}^T(\mathbf{\xi}_{2_{\mathbf{m}_{(2)}}}),
 \end{align}
where $\mathcal{P}_{\mathbf{m}_{(2)}\mathbf{m}_{(3)}}$ is the interpolation coefficient between $\mathbf{\xi}_{2_{\mathbf{m}_{(2)}}}$ and $\mathbf{\xi}_{3_{\mathbf{m}_{(3)}}}$ and the superscript $T$ implies matrix transposition.
The interpolation strategy of the second form in \eqref{eq:51} is also called anterpolation.
\par
Inserting \eqref{eq:51} into \eqref{eq:50}, we have
\begin{align}
\label{eq:53}
 \Phi_\sigma^{FMM}(\mathbf{x}_i-\mathbf{x}_j)
        =&\sum_{\mathbf{m}_{(2)}=1}^{\mathbf{M}_{(2)}}\omega_{2_{\mathbf{m}_{(2)}}}\sum_{\mathbf{m}_{(3)}=1}^{\mathbf{M}_{(3)}}\mathcal{P}_{\mathbf{m}_{(3)}}^T(\mathbf{\xi}_{2_{\mathbf{m}_{(2)}}})e^{i\mathbf{\xi}_{3_{\mathbf{m}_{(3)}}}\mathbf{r}_{ia_{3}}}e^{i\mathbf{\xi}_{2_{\mathbf{m}_{(2)}}}\mathbf{r}_{a_3a_2}}\mathcal{C}(\mathbf{\xi}_{3_{\mathbf{m}_{(2)}}})\\
\nonumber         &e^{i\mathbf{\xi}_{2_{\mathbf{m}_{(2)}}}\mathbf{r}_{a_2b_2}}e^{i\mathbf{\xi}_{2_{\mathbf{m}_{(2)}}}r_{b_2b_3}}\sum_{\mathbf{m}_3=1}^{\mathbf{M}_{(3)}}\mathcal{P}_{\mathbf{m}_{(3)}}(\mathbf{\xi}_{2_{\mathbf{m}_{(2)}}})e^{i\mathbf{\xi}_{3_{\mathbf{m}_{(3)}}}\mathbf{r}_{b_3j}}\\
\nonumber=&\underbrace{\sum_{\mathbf{m}_{(3)}=1}^{\mathbf{M}_{(3)}}\omega_{3_{\mathbf{m}_{(2)}}}e^{i\mathbf{\xi}_{3_{\mathbf{m}_{(3)}}}\mathbf{r}_{ia_{3}}}\sum_{\mathbf{m}_{(2)}=1}^{\mathbf{M}_{(2)}}\mathbf{\omega}_{2_{\mathbf{m}_{(2)}}}/\omega_{3_{\mathbf{m}_{(2)}}}\mathcal{P}_{\mathbf{m}_{(3)}}^T(\mathbf{\xi}_{2_{\mathbf{m}_{(2)}}})e^{i\mathbf{\xi}_{2_{\mathbf{m}_{(2)}}}\mathbf{r}_{a_3a_2}}}_{\text{Downsweep}}\\
\nonumber         &\underbrace{\mathcal{C}(\mathbf{\xi}_{3_{\mathbf{m}_{(2)}}})e^{i\mathbf{\xi}_{2_{\mathbf{m}_{(2)}}}\mathbf{r}_{a_2b_2}}}_{\text{Coupling}}\underbrace{e^{i\mathbf{\xi}_{2_{\mathbf{m}_{(2)}}}\mathbf{r}_{b_2b_3}}\sum_{\mathbf{m}_{(3)}=1}^{\mathbf{M}_{(3)}}\mathcal{P}_{\mathbf{m}_{(3)}}(\mathbf{\xi}_{2_{\mathbf{m}_{(2)}}})e^{i\mathbf{\xi}_{3_{\mathbf{m}_{(3)}}}\mathbf{r}_{b_3j}}}_{\text{Upsweep}}.
 \end{align}
Let the matrices $\mathbf{U}^{l}$, $\mathbf{K}^{l}$, and $\mathbf{V}^{l}$ identify as $L2P (or~ L2L)$, $M2L$, and $P2M (or~M2M)$ operators, respectively. The matrix-vector product
will be written as
\begin{align}\label{eq:54}
\sum_{j=1}^N\lambda_j\Phi(\mathbf{x}_i-\mathbf{x}_j)&=A\mathbf{\lambda}=A^{near}\mathbf{\lambda}+A^{far}\mathbf{\lambda}\\
\nonumber &=\underbrace{A^{diagonal}\mathbf{\lambda}+A^{near-diagonal}\mathbf{\lambda}}_{\text{sparse matrix-vector product}}+\underbrace{\mathbf{U}^l\mathbf{K}^l\mathbf{V}^l\mathbf{\lambda}}_{\text{MLFMM product}}
\end{align}
\subsubsection{Interpolation and Anterpolation}
In this section, we discuss an approach for interpolating the multipole expansion up the tree and anterpolating local expansion down the tree. 
In general, the truncation number at different levels satisfies $\mathbf{M}_{(l)}< \mathbf{M}_{(l-1)}$. But through scaling property of the Fourier transform, we will have $\mathbf{M}_{(l)}=\mathbf{M}_{(l-1)}$.
\par
Here, we introduce scaling property of Fourier transform briefly. For a function $f(x)$, $x\in\Omega\subset\Bbb{R}$, its Fourier transform is $\mathcal{F}(\xi)$.
For any $s>0$, then
\begin{equation}
\mathcal{F}(f(sx))= \int_\Omega f(sx)e^{-ix\xi}dx
                  = \frac{1}{s}\int_{s\Omega} f(sx)e^{-isx\frac{\xi}{s}}d(sx)
                  = \frac{1}{s}\mathcal{F}(\frac{\xi}{s}).
\end{equation}
Suppose $f(x)$ is a band-limited function with the bandwidth $\sigma>0$, the function $f(x)$ can be approximated by 
\begin{equation}
f(\mathbf{x})=\dfrac{1}{(2\pi)^d}\sum_{\mathbf{m}=1}^\mathbf{M}\omega_m\mathcal{F}(\xi_m)e^{i\mathbf{\xi}_\mathbf{m}\mathbf{x}}+o(\mathbf{\xi}_\mathbf{M}),
\end{equation}
thus,
\begin{equation}
f(s\mathbf{x})=\dfrac{1}{(2\pi)^d}\sum_{\mathbf{m}=1}^\mathbf{M}\frac{\mathbf{\omega}_\mathbf{m}}{s}\mathcal{F}(\frac{\mathbf{\xi}_\mathbf{m}}{s})e^{i\frac{\xi_m}{s}\mathbf{x}}+o(\dfrac{\mathbf{\xi}_\mathbf{M}}{s}).
\end{equation}
Therefore, the following Lemma holds.
\begin{lemma}
 Suppose that $\varepsilon_{\mathbf{M}}>0$ is a real number, such that
 \begin{equation}\label{eq:30}
  |\Phi_\sigma(\mathbf{x}-\mathbf{y})-\dfrac{1}{(2\pi)^2}\sum_{\mathbf{m}=1}^{\mathbf{M}}\mathbf{\omega}_{\mathbf{m}}\widehat{\Phi}(\mathbf{\xi}_\mathbf{m})e^{i\mathbf{\xi}_{\mathbf{m}}(\mathbf{x}-\mathbf{y})}|<\varepsilon_{\mathbf{M}},
 \end{equation}
 for all $\mathbf{y}$ is contained inside a square $A$ of length $1$, and $\mathbf{x}$ is an arbitrary point belonging to the interaction region of $A$.
 Suppose further $s>0$ is a real number, with $\overline{\mathbf{y}}\in A_s,$ the length of the square $A_s$ is $s$, and the function $\psi$ are defined by the formula
 \begin{equation}\label{eq:31}
  \psi(\overline{\mathbf{x}})=\dfrac{1}{s}\widehat{\Phi}(\frac{\mathbf{\xi}_{\mathbf{m}}}{s})e^{i\mathbf{\xi}_{\mathbf{m}}\frac{(\overline{\mathbf{x}}-\overline{\mathbf{y}})}{s}}  
 \end{equation}
for all $\overline{\mathbf{x}}\in \R^2$. Finally, suppose that
\begin{equation}\label{eq:32}
 u(\overline{\mathbf{x}})=\sum_{j=1}^N\lambda_j\Phi_\sigma(\overline{\mathbf{x}}-\overline{\mathbf{x}_j}),
\end{equation}
is the potential field located at points $\overline{\mathbf{x}}_1,\overline{\mathbf{x}}_2,\cdots,\overline{\mathbf{x}}_N$ inside the square $A_s$. Then for any $\overline{\mathbf{x}}$ belongs to
the interaction region of $A_s$,
\begin{equation}\label{eq:33}
 |u(\overline{\mathbf{x}})-\dfrac{1}{(2\pi)^2}\sum_{\mathbf{m}=1}^{\mathbf{M}}\mu_{\mathbf{m}}\psi(\overline{\mathbf{x}})|<\dfrac{\varepsilon_{\mathbf{M}}}{s}\sum_{j=1}^N|\lambda_j|,
\end{equation}
with 
\begin{equation}\label{eq:34}
 \mu_{\mathbf{m}}=\mathbf{\omega}_{\mathbf{m}}\sum_{j=1}^N\lambda_je^{i\mathbf{\xi}_{\mathbf{m}}\frac{\overline{\mathbf{y}}-\overline{\mathbf{x}}_j}{s}}
\end{equation}
for all~$\mathbf{m}=1,2,\cdots,\mathbf{M}$.
\end{lemma}
For a three-level FMM, because the length of parent square is twice that of child square, i.e., $s=2$, it is easy to find that $\mathbf{\xi}_{2_{\mathbf{m}_{(2)}}}=\frac{\mathbf{\xi}_{3_{\mathbf{m}_{(3)}}}}{2}$ and $\mathbf{\omega}_{2_{\mathbf{m}_{(2)}}}=\frac{\omega_{3_{\mathbf{m}_{(3)}}}}{2}$.
Suppose further $\mathbf{m}$ is even, i.e. $\mathbf{m}_{(3)}=2\mathbf{n}, \mathbf{n}=1,2,\cdots,\mathbf{M}_{(3)}$, \eqref{eq:53} can be written as
\begin{align}\label{eq:60}
 &\Phi_\sigma^{FMM}(\mathbf{x}_i-\mathbf{x}_j)\\
\nonumber &=
 \sum_{\mathbf{m}_{(3)}=\mathbf{2n}}^{\mathbf{2M}_{(3)}}\mathbf{\omega}_{3_{\mathbf{m}_{(3)}}}e^{i\mathbf{\xi}_{3_{\mathbf{m}}}\mathbf{r}_{ia_3}}
 \sum_{\mathbf{n}=1}^{\mathbf{M}_{(3)}}\frac{1}{2}\mathcal{P}_{\mathbf{m}_{(3)}}^T(\mathbf{\xi}_{3_{\mathbf{n}}})e^{i\mathbf{\xi}_{3_{\mathbf{n}}}\mathbf{r}_{a_3a_2}}
 \mathcal{C}(\mathbf{\xi}_{3_{\mathbf{n}}})e^{i\mathbf{\xi}_{3_{\mathbf{n}}}\mathbf{r}_{a_2b_2}}e^{i\mathbf{\xi}_{3_{\mathbf{n}}}\mathbf{r}_{b_2b_3}}\\
\nonumber &\sum_{\mathbf{m}_{(3)}=\mathbf{2n}}^{\mathbf{2M}_{(3)}}\mathcal{P}_{\mathbf{m}_{(3)}}(\mathbf{\xi}_{3_{\mathbf{n}}})e^{i\mathbf{\xi}_{3_{\mathbf{m}_{(3)}}}\mathbf{r}_{b_3j}}.
 \end{align}
As described above, according to the Weierstrass approximation theorem, the continuous function $e^{i\frac{\xi}{2}x}$ can be approximated by polynomials. 
In this paper, we consider Lagrange interpolation method because the local interpolation method is fast and has simple error analysis, i.e.,
given and $x\in [-a,a]$, for $\varepsilon>0$, there exists a $K>0$ such that
 \begin{equation}
  \max_{\xi\in[-\sigma,\sigma]}|e^{i\frac{\xi}{2}x}-\sum_{k=1}^K\mathcal{P}_{\xi_k}(\frac{\xi}{2})e^{i\xi_k x}|<\varepsilon,
 \end{equation}
 where the Lagrange interpolation operator
 \begin{equation}
 \mathcal{P}_{\xi_k}(\frac{\xi}{2})=\prod_{l=1,k\neq l}^K\dfrac{\frac{\xi}{2}-\xi_k}{\xi_l-\xi_k}.
 \end{equation}
The truncation errors for various values of $K$ in infinity norm corresponding to $a=1$ are summarized in Table \ref{tab:halfwaveerror}.
\begin{table}[H]
\centering
\begin{tabular}{ccccc}
\toprule
$K$                &$5$             &$6$            &$7$              &$8$\\

$\varepsilon$                                                                                                                                                                                                                                                                                   &$0.0830$        &$0.0211$       &$0.0048$         &$5.7199e-04$  \\
\midrule
$K$       &$9$           &$10$           &$11$            &$12$\\
 $\varepsilon$&$8.1828e-05$  &$1.3111e-05$   &$1.9350e-06$    &$1.5142e-07$\\ 
\bottomrule
\end{tabular}
\caption{Lagrange interpolation error}
\label{tab:halfwaveerror}
\end{table}
\section{Numerical Examples}
We tested the method using three different kernels that have different properties:
\begin{itemize}
 \item Inverse Multiquadric (IMQ), $\dfrac{1}{\sqrt{1+r^2}}$. Monotonically decaying global radial basis function.
 \item Multiquadric (MQ), $\sqrt{1+r^2}$. Monotonically increasing global radial basis function.
 \item Wendland's function, $(1-r,0)_{+}^3(3r+1)$. Compactly supported radial basis function.
\end{itemize}
In this paper, we use the mollification as the approximated to replace the original function. Figure 3 presents three different original functions, 
corresponding approximation functions and Fourier transform. We have used $[0,1]$ as the computing domain. For the compactly supported function, we scaled the compactly
supported radius so that its Fourier transform is not a constant. As expected frequencies of three functions decay and tend to zeros in finite intervals. This result verified
our approach is valid.
\begin{figure}%
\centering
   \subfloat[$\Phi(r)=\frac{1}{\sqrt{r^2+1}}$]{\includegraphics[width=0.32\textwidth]{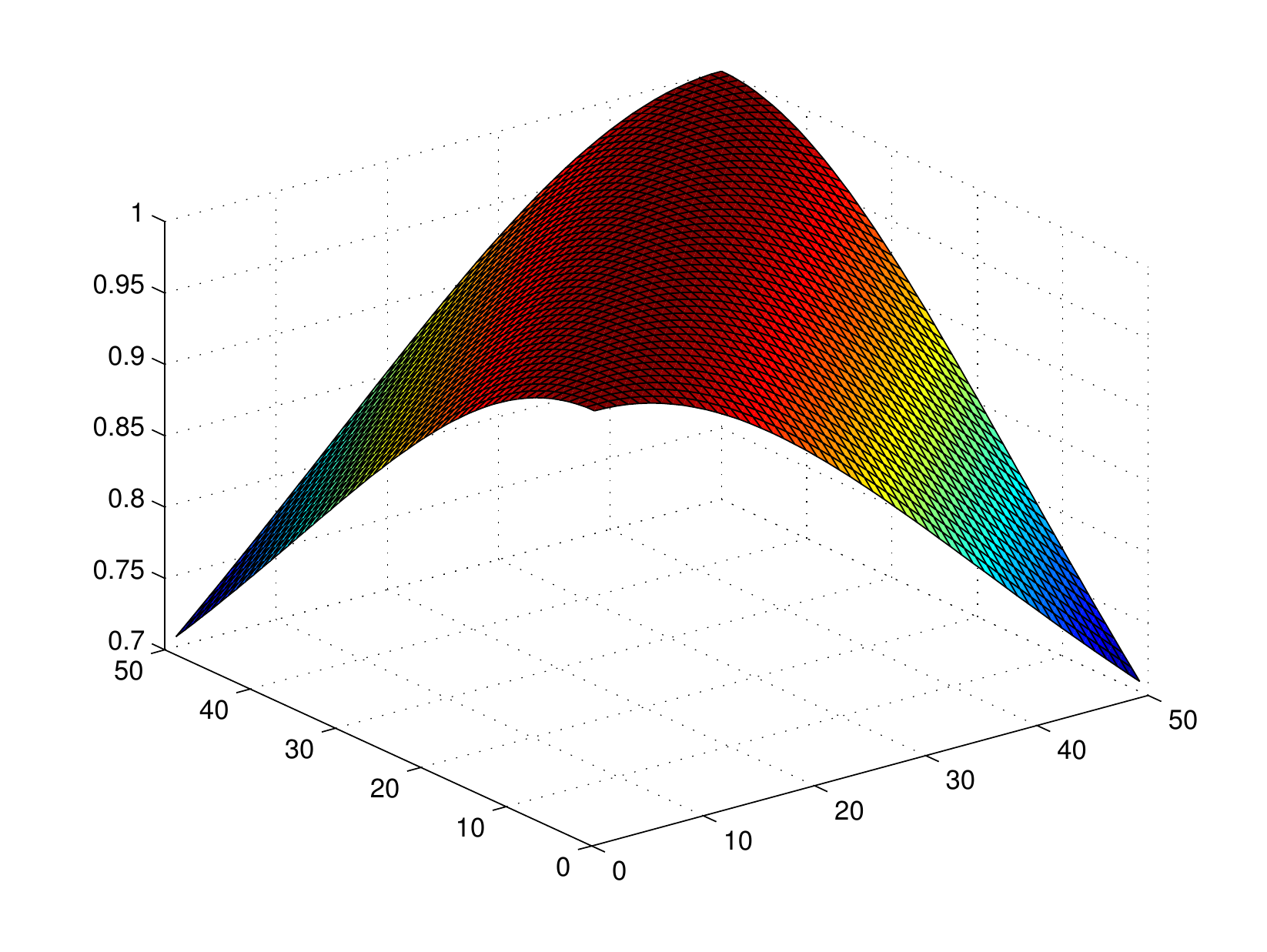}} 
   \subfloat[$\Phi_\sigma^{FMM}(r)$]{\includegraphics[width=0.32\textwidth]{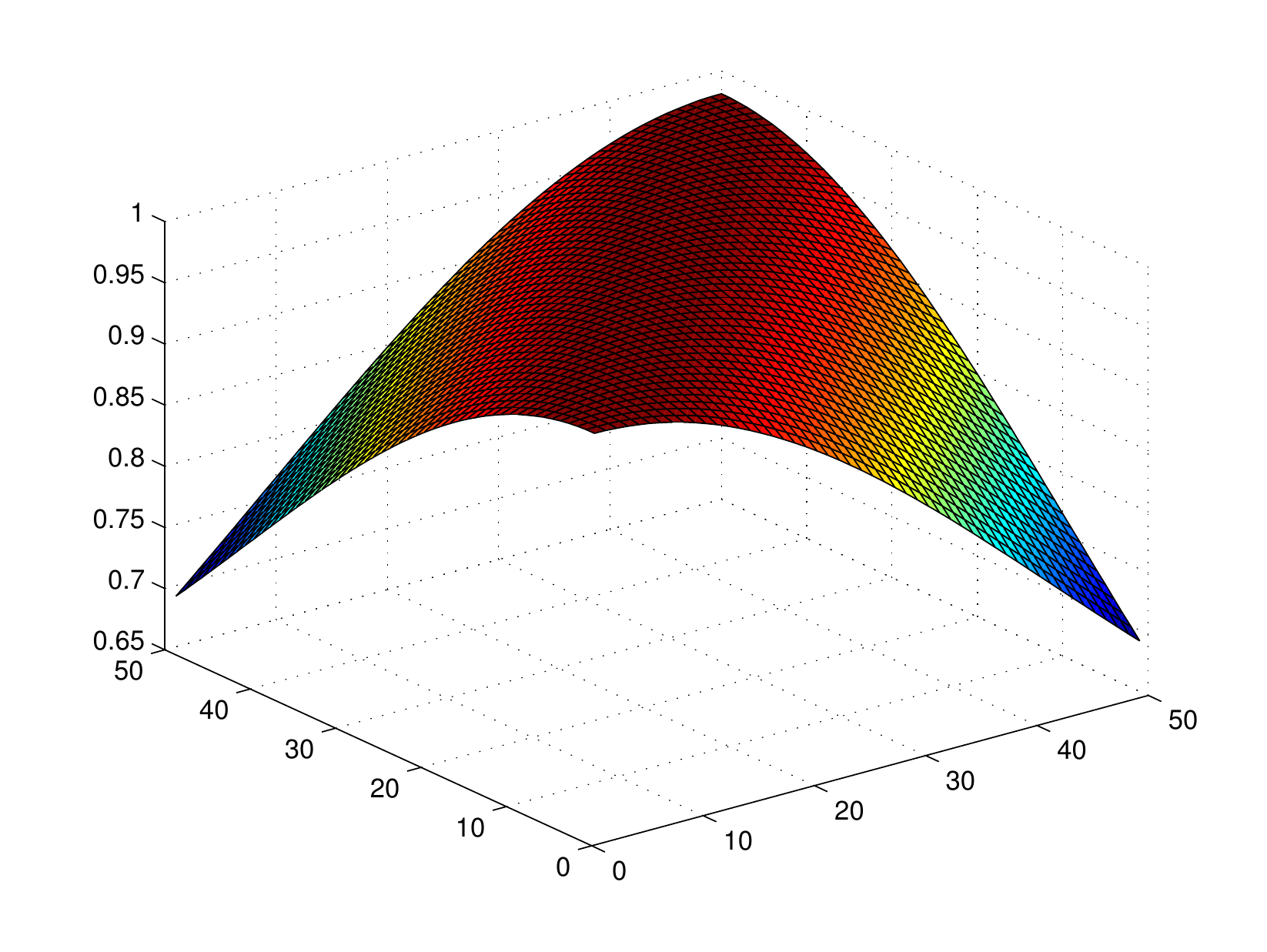}}
   \subfloat[$\widehat{\Phi}_\sigma^{FMM}$]{\includegraphics[width=0.32\textwidth]{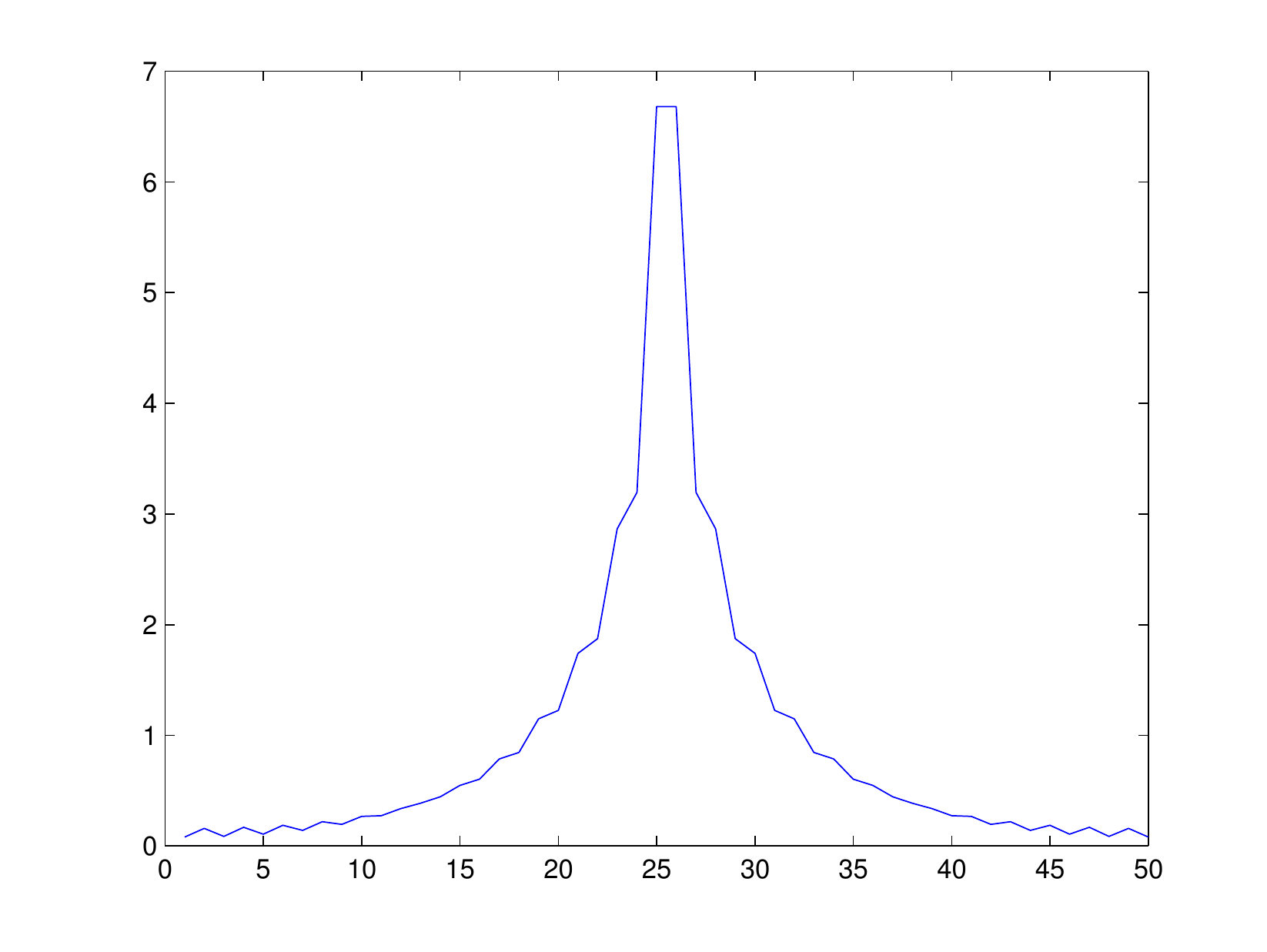}}

   \subfloat[$\Phi(r)=\sqrt{r^2+1}$]{\includegraphics[width=0.32\textwidth]{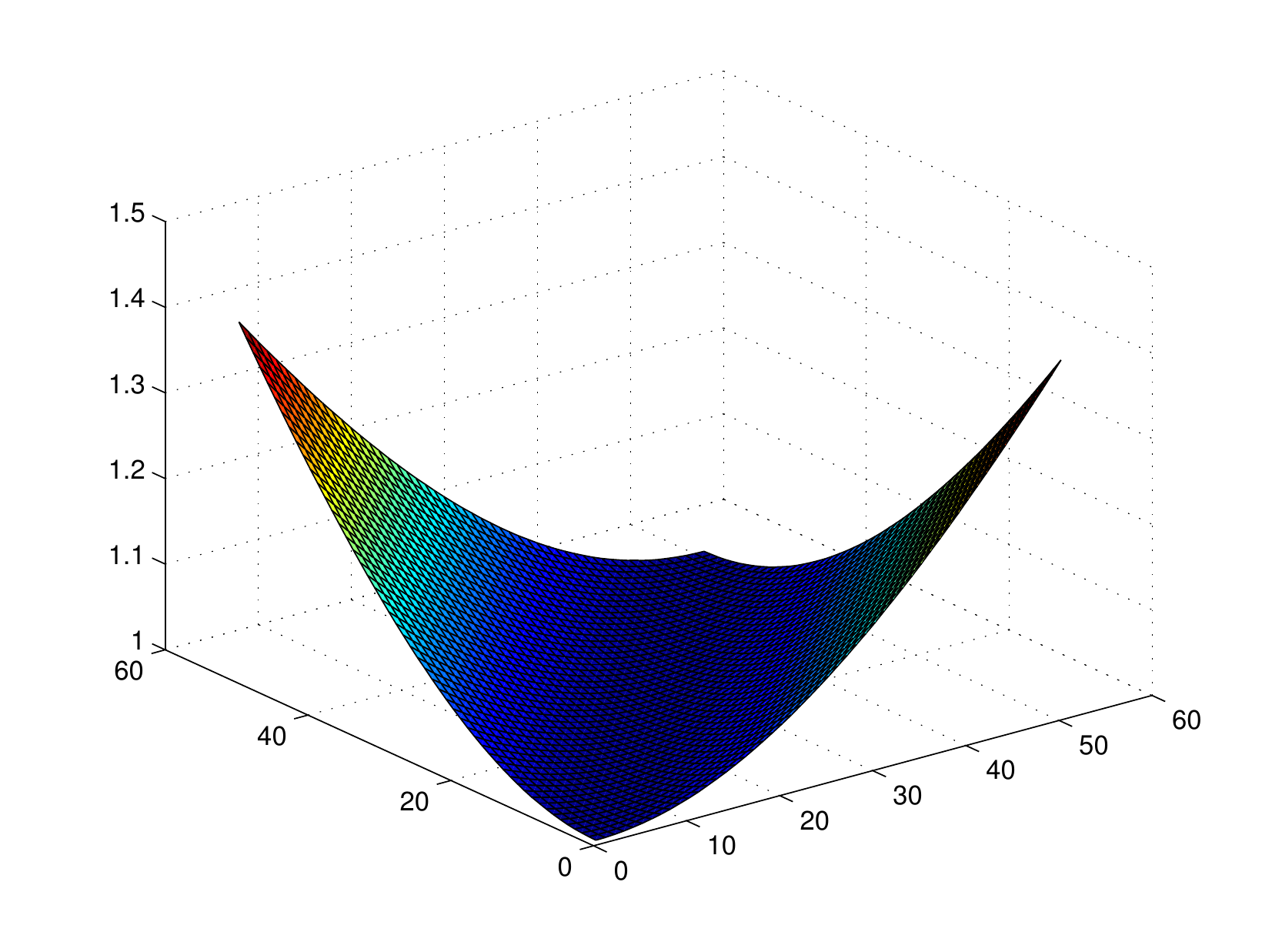}} 
   \subfloat[$\Phi_\sigma^{FMM}(r)$]{\includegraphics[width=0.32\textwidth]{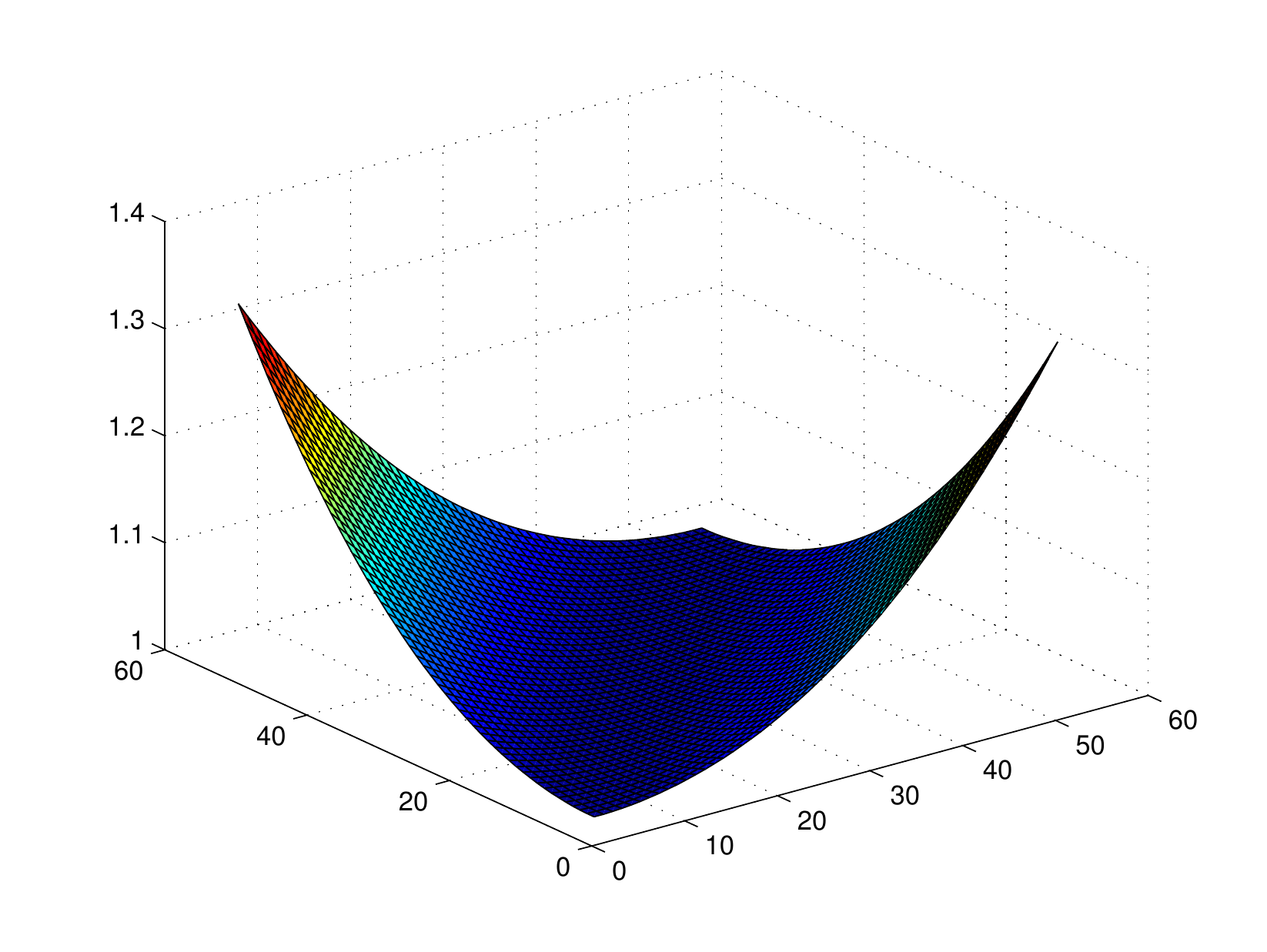}}
   \subfloat[$\widehat{\Phi}_\sigma^{FMM}$]{\includegraphics[width=0.32\textwidth]{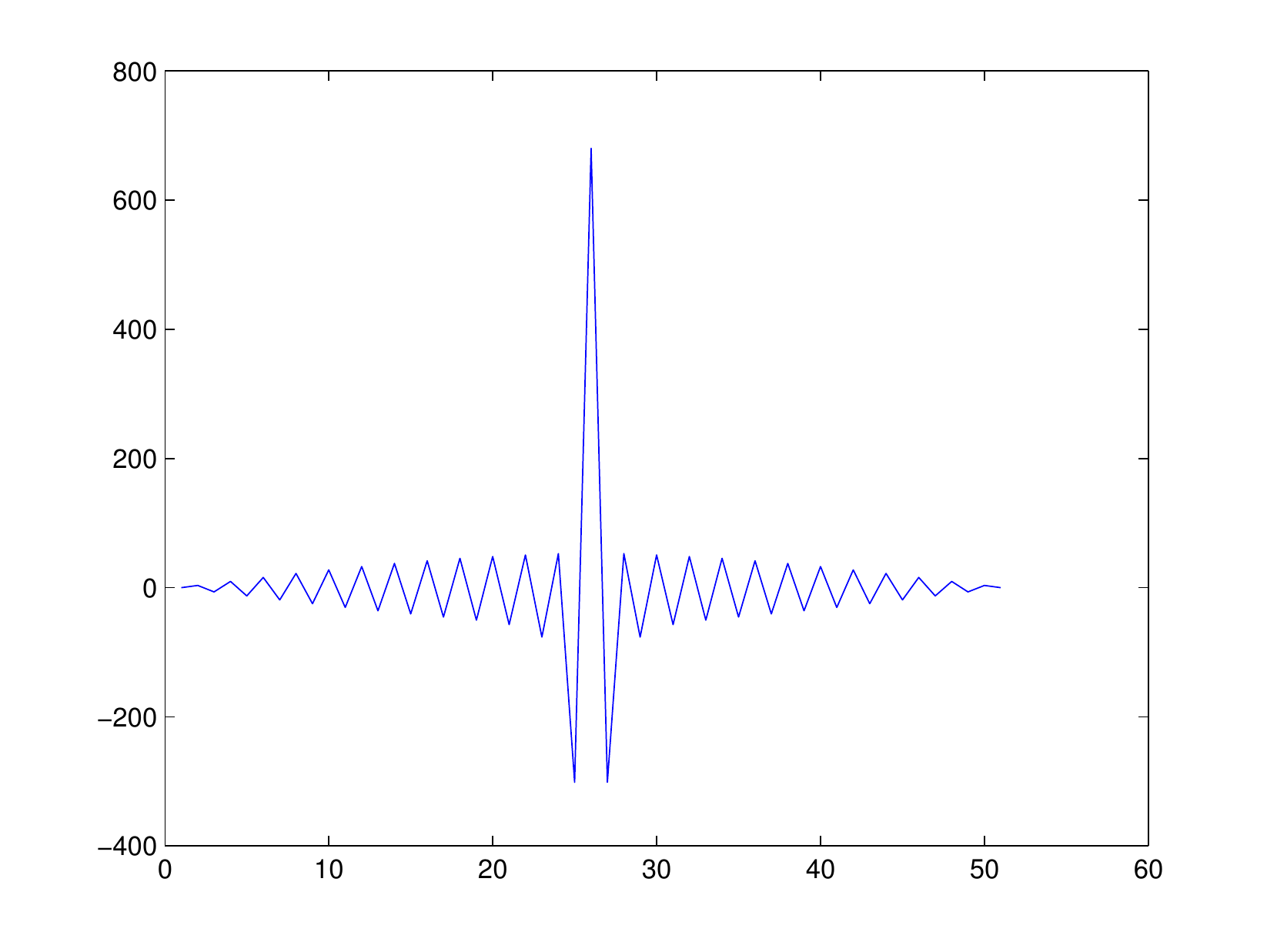}}

   \subfloat[$\Phi(r)=(1-\varepsilon r,0)_{+}^3(3\varepsilon r+1),\varepsilon=0.5$]{\includegraphics[width=0.32\textwidth]{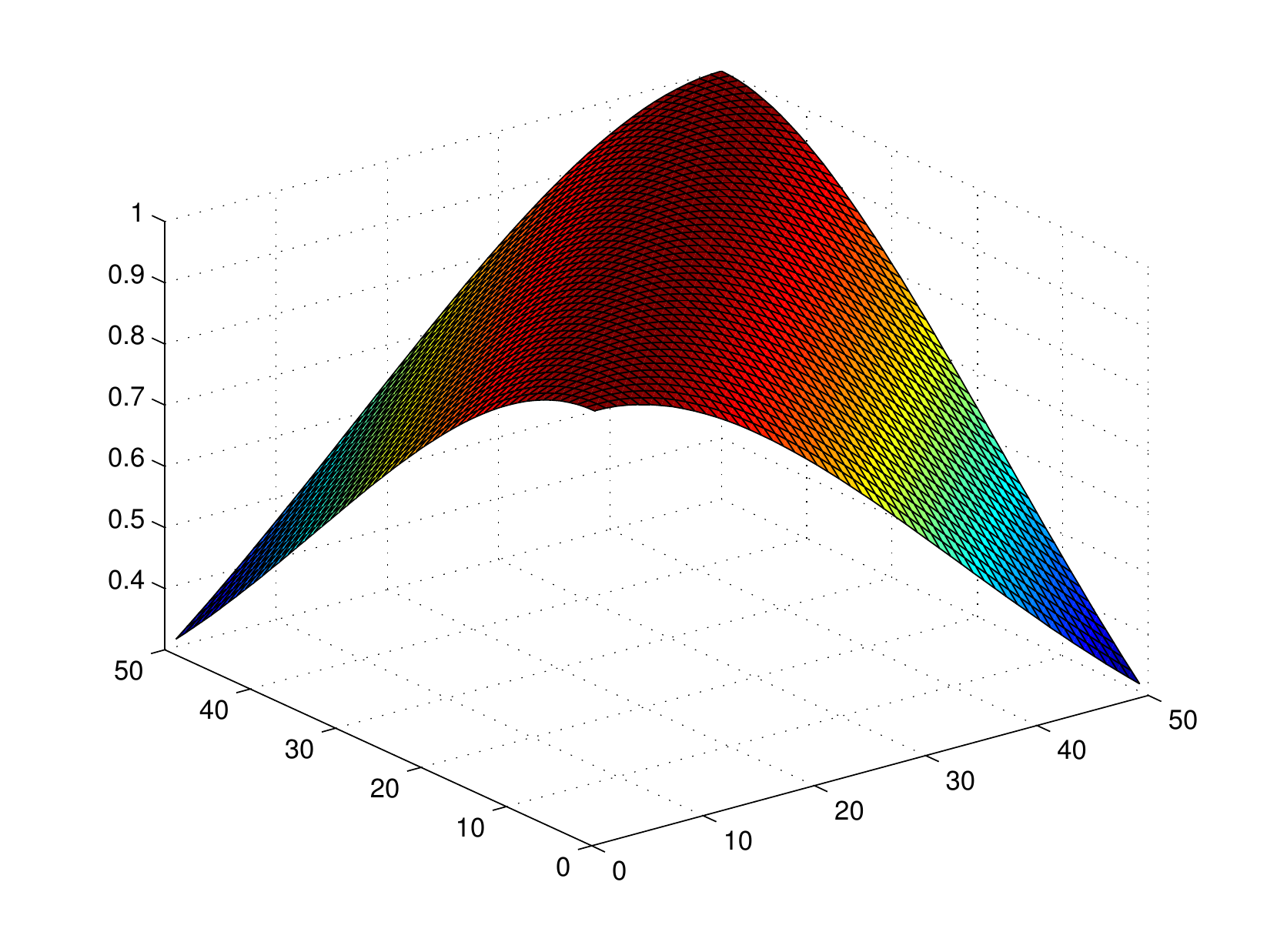}} 
   \subfloat[$\Phi_\sigma^{FMM}(r)$]{\includegraphics[width=0.32\textwidth]{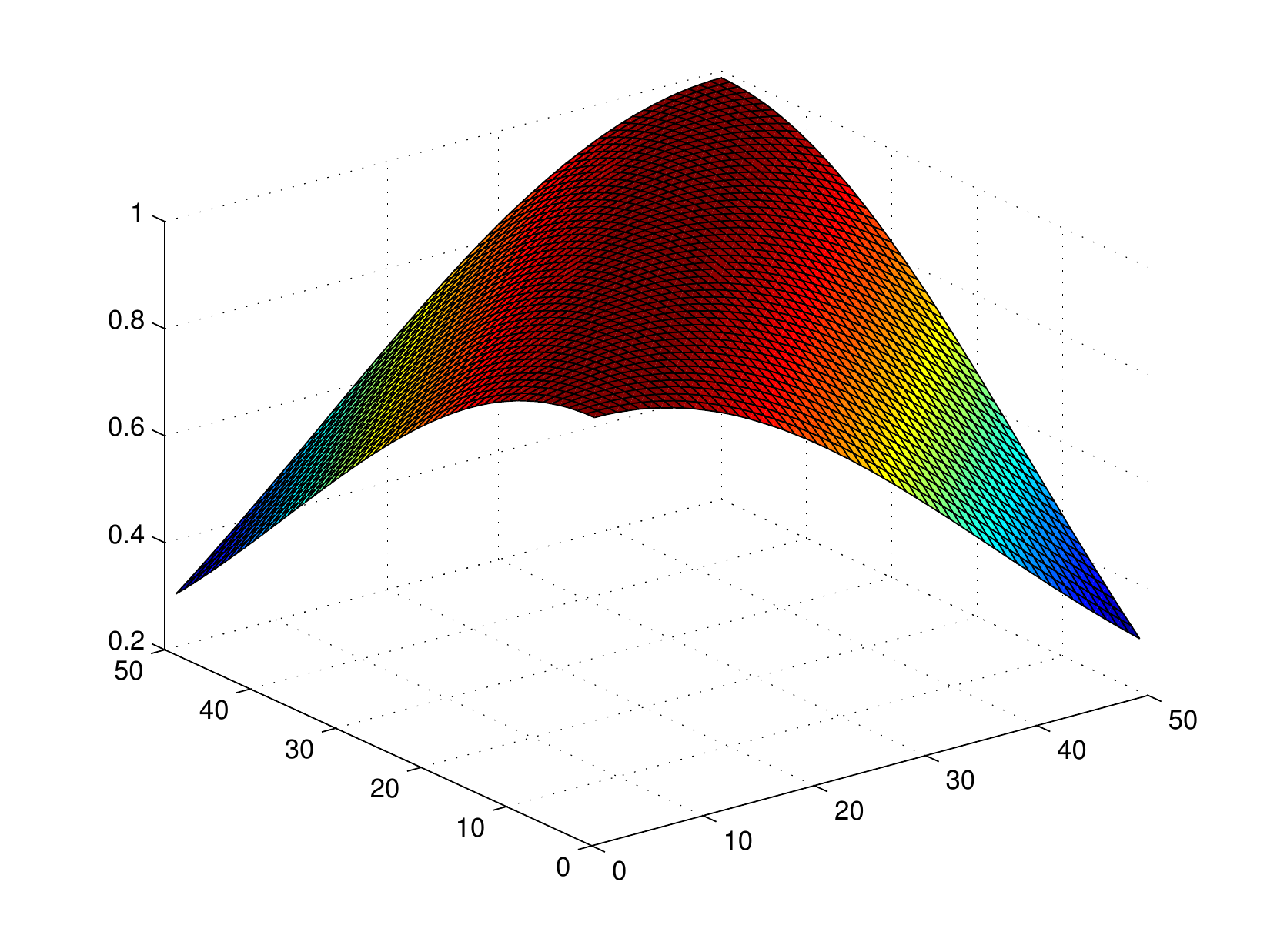}}
   \subfloat[$\widehat{\Phi}_\sigma^{FMM}$]{\includegraphics[width=0.32\textwidth]{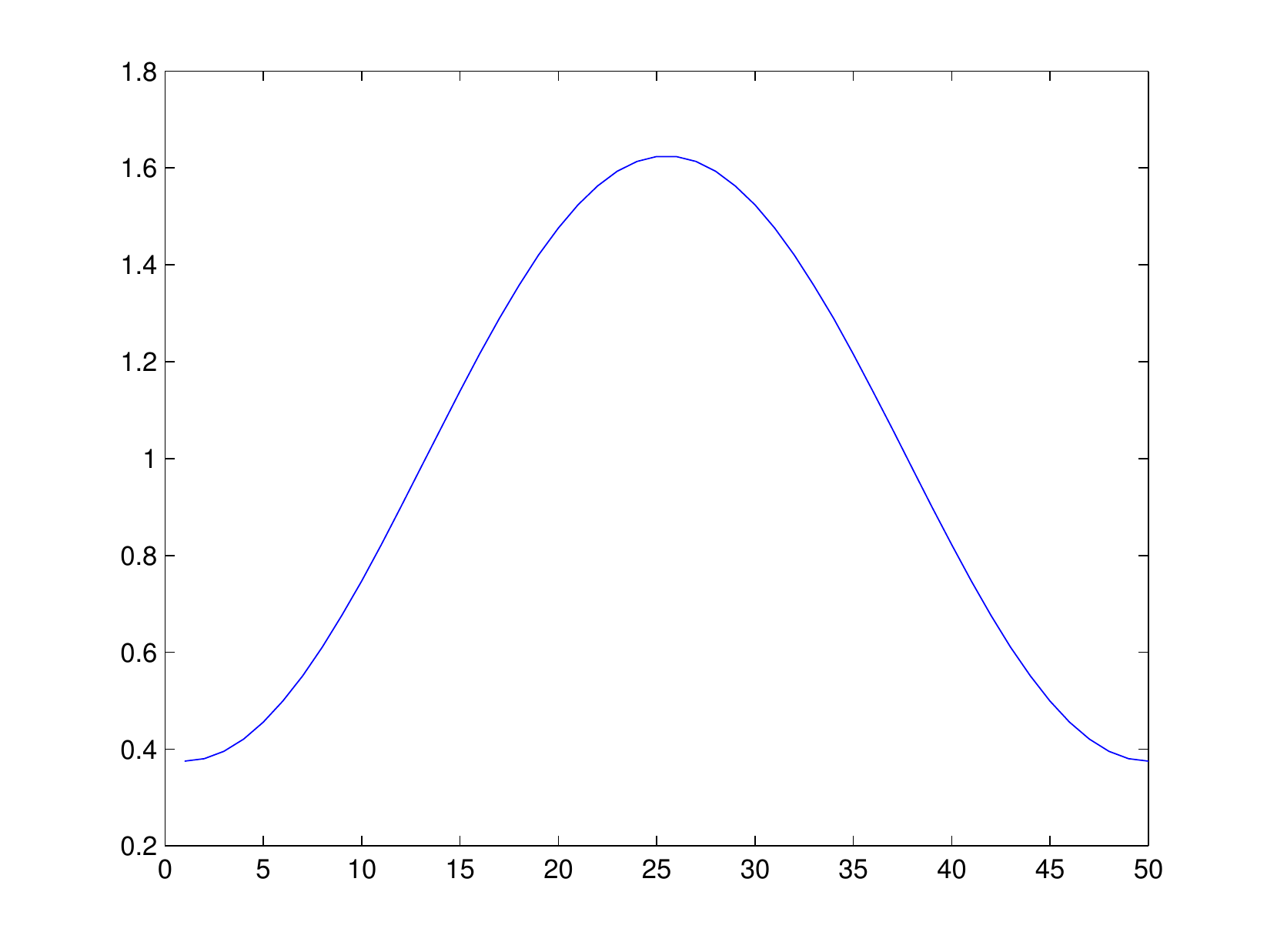}}

   \caption{Left: three radial basis functions. Middle: low-rank representations of three functions after introducing mollifiers. Right: Fourier transforms of low-rank representations.}%
   \label{fig:BLapp}%
\end{figure} 

Next we are interested in the convergence of the scheme in this paper. For simplicity we show only results corresponding to the one level FMM. In Figure 4, $R$ is the length
of the cluster. The FMM expansion is only applied to all clusters that are well separated and others are computed directly without any accelerated scheme.

\begin{figure}%
\centering
   \subfloat[$\frac{1}{\sqrt{r^2+1}}$]{\includegraphics[width=0.32\textwidth]{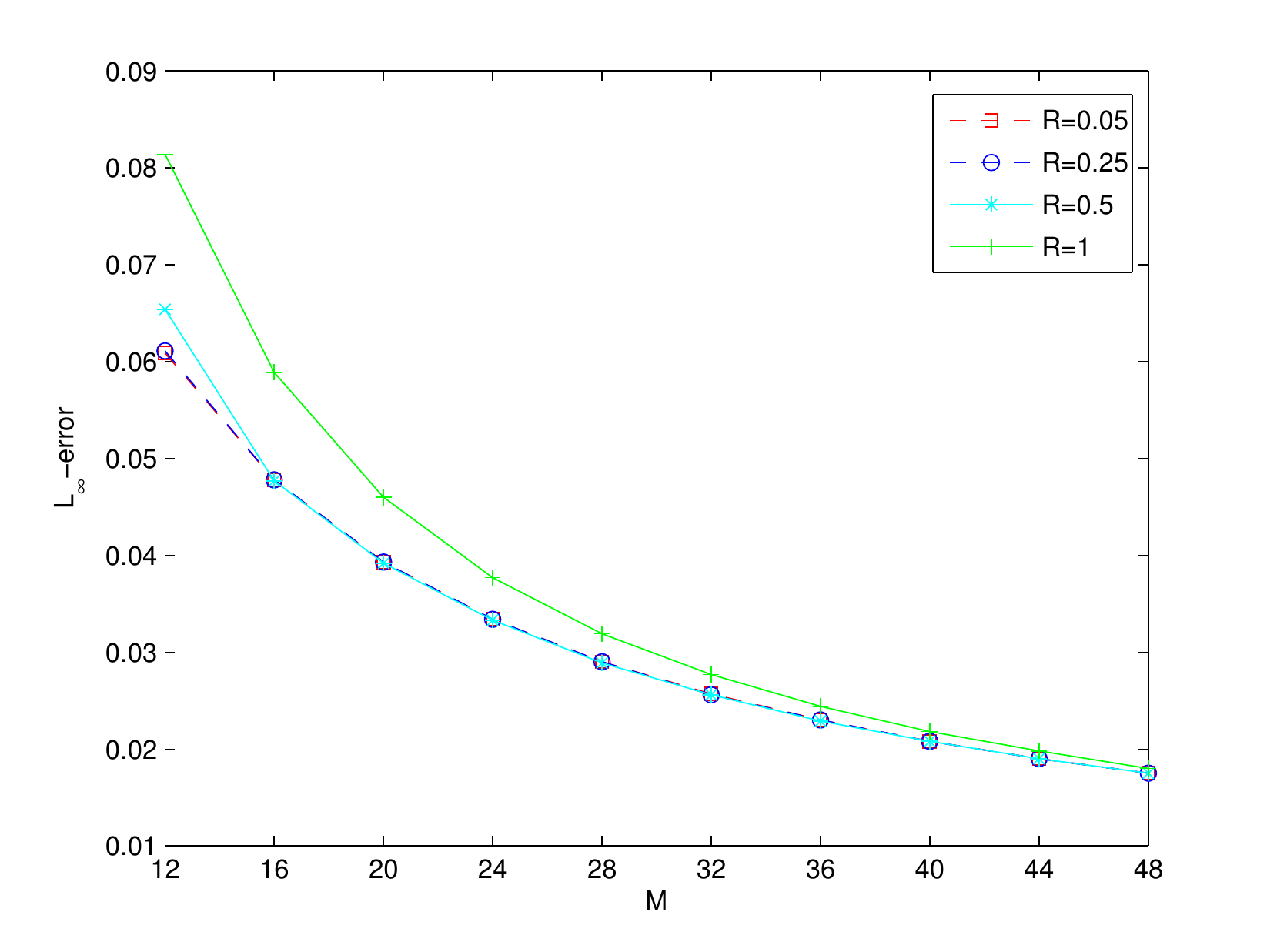}} 
   \subfloat[$\sqrt{r^2+1}$]{\includegraphics[width=0.32\textwidth]{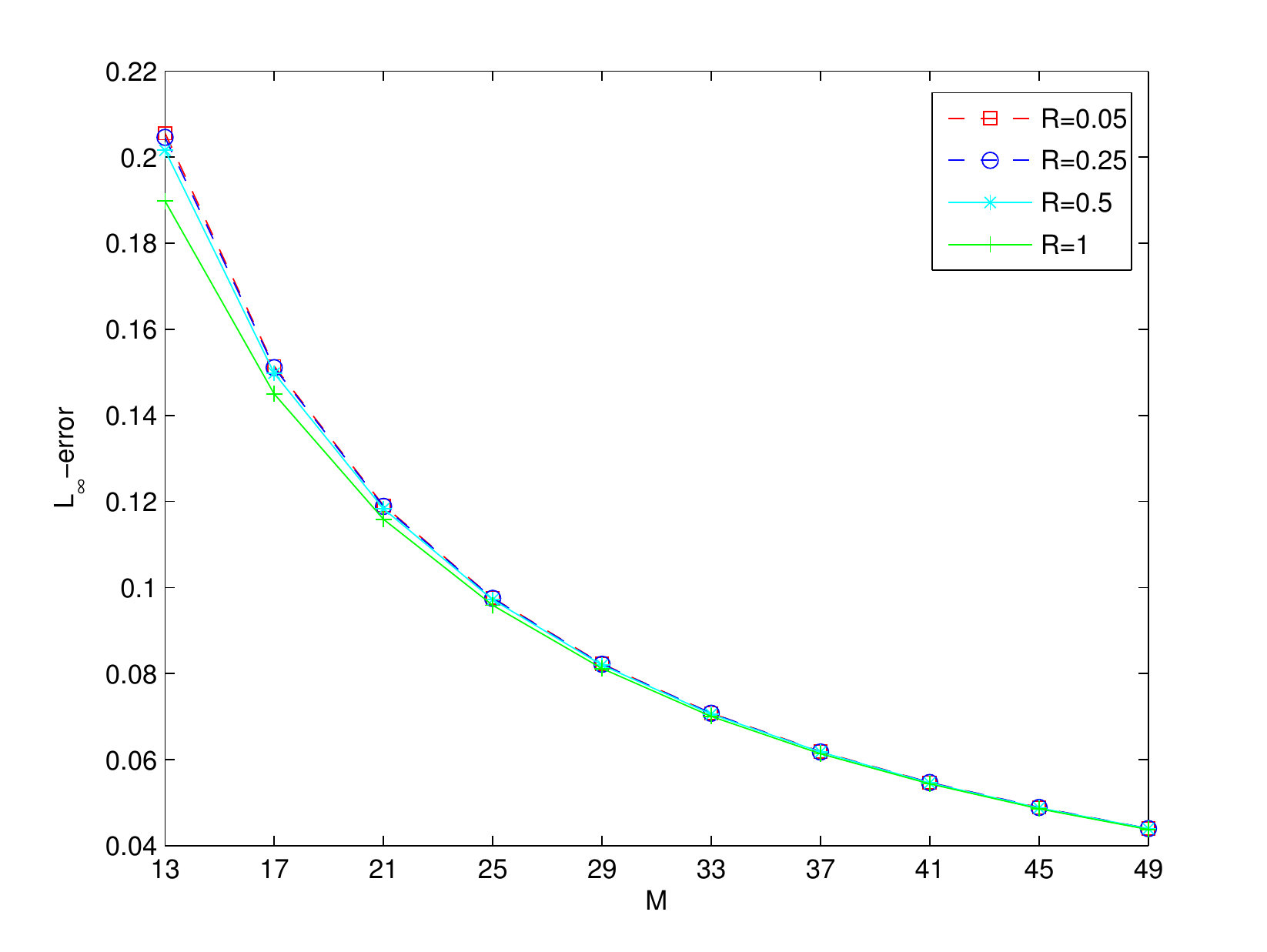}}
   \subfloat[$(1-\varepsilon r,0)_{+}^3(3\varepsilon r+1),\varepsilon=0.5$]{\includegraphics[width=0.32\textwidth]{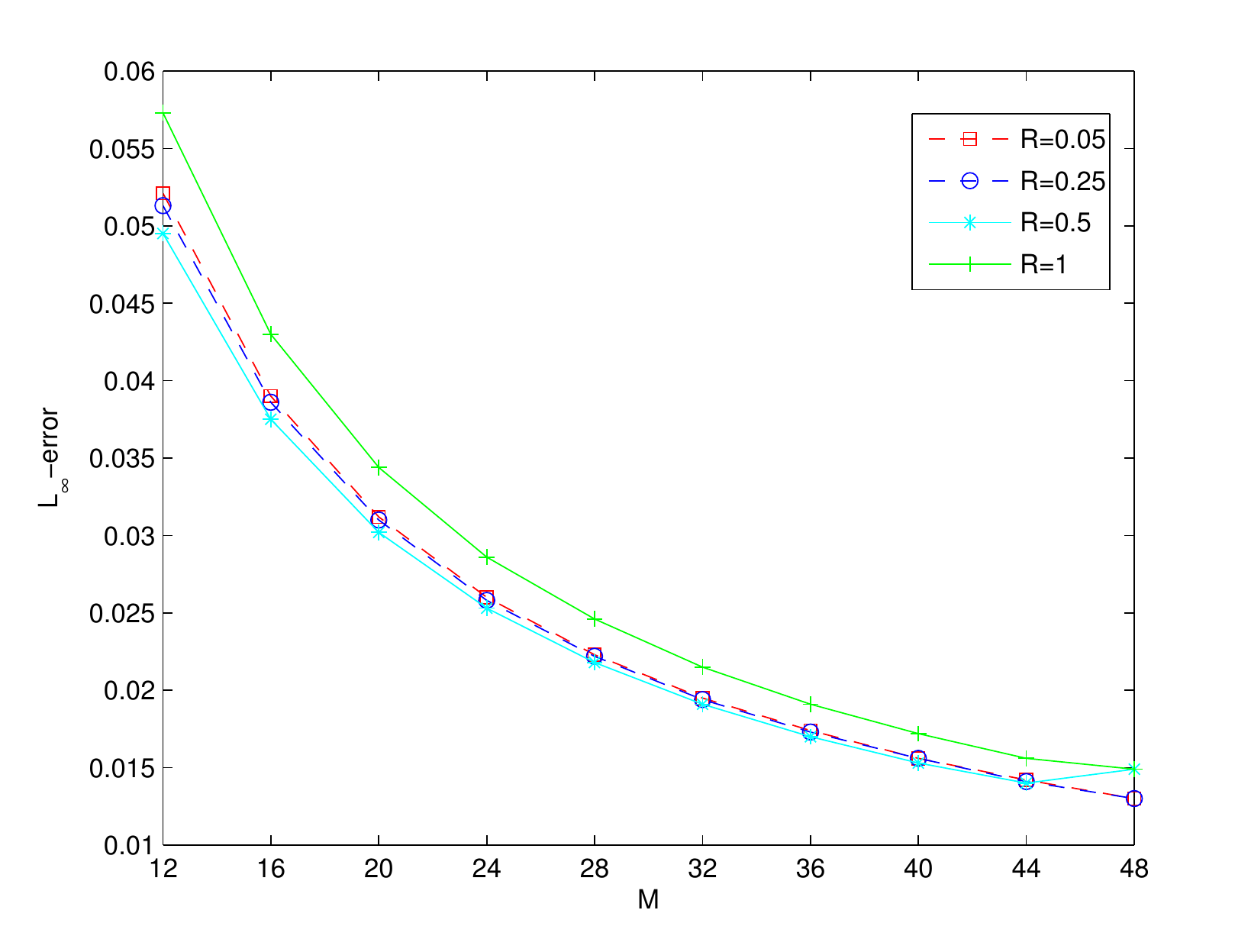}}

   \caption{$L_{\infty}$-errors for three functions with various values $R$ and truncation numbers and the error of the algorithm dependends on the truncation number $M$.}%
   \label{fig:BLapperr}%
\end{figure} 

\section{Conclusion}
We gave a framework to improve smooth kernels so that the range of applicability of the fast multipole method can be extended. This 
approach speeds up the traditional FMMs because the M2L translation operators are diagonal. The algorithm shares similarities with 
high-frequency fast multipole methods.
\par
The approach presented here works independently of the kernel as long as it is smooth. According to a fundamental principle of the Fourier transform: smooth functions have Fourier transforms that decay rapidly to zeros at infinity (no energy at highest frequency). Based on this principle, we introduced a suitable mollifier to improve the smoothness of radial basis functions and radial basis functions are then replaced by these smoother kernels for obtaining low-rank representations easily. These smoother kernels share many properties with radial basis functions.

We also gave Sobolev-type error estimate and stability analysis for interpolation by smoother kernels. Our numerical results have shown that the proposed method 
is convergent.

\bibliographystyle{siam}
\bibliography{mydata,data}
\end{document}